%% file: main.tex
\documentclass[12pt]{article}
\usepackage{amsmath, amsfonts, amssymb, tikz-cd, mathtools}
\usepackage{graphicx, multicol}
\usepackage{caption, subcaption}
\usepackage{enumitem} 
\usepackage[sc]{mathpazo}
\usepackage[T1]{fontenc}
\usepackage[utf8]{inputenc}

\usepackage{geometry}
\usepackage{amsthm}
\usepackage[most]{tcolorbox}

\numberwithin{equation}{section}

\newtheorem{theorem}{Theorem}[section]
\newtheorem{lemma}[theorem]{Lemma}
\newtheorem{proposition}[theorem]{Proposition}
\newtheorem{corollary}[theorem]{Corollary}

\theoremstyle{definition}
\newtheorem{definition}[theorem]{Definition}

\theoremstyle{remark}
\newtheorem{remark}[theorem]{Remark}

\title{Connectedness of the Gromov boundary of fine curve graphs}
\author{{\sc Yusen Long and Dong Tan}}
\date{}

\definecolor{upsPurp}{RGB}{99,0,60}
\definecolor{upsBlue}{RGB}{0,78,125}
\definecolor{upsCyan}{RGB}{0,148,181}

\usepackage{hyperref}
\hypersetup{
     colorlinks=true,
     linkcolor=upsBlue,
     citecolor = upsPurp,      
     urlcolor= upsCyan,
     }

\usepackage[
backend=biber,
style=alphabetic,
sorting=nyt,
maxbibnames=99
]{biblatex}
\DeclareMathOperator{\Isom}{Isom}
\DeclareMathOperator{\diam}{diam}
\DeclareMathOperator{\Id}{Id}
\DeclareMathOperator{\homeo}{Homeo}

\DeclareMathOperator{\NC}{\mathcal{NC}}
\DeclareMathOperator{\NCo}{\mathcal{NC}^\dagger_{o,\pitchfork}}
\DeclareMathOperator{\Co}{\mathcal{C}^\dagger_{o,\pitchfork}}
\DeclareMathOperator{\So}{\mathcal{S}_{o,\pitchfork}}

\usepackage{titlesec}
\titleformat{\subsection}[runin]{\bf}{\thesubsection.}{3pt}{}
\renewcommand\thesubsection{\thesection.\Alph{subsection}}

\begin{document}

\maketitle

\begin{abstract}
The fine curve graph was introduced to study homeomorphism group of surfaces. In this paper we study the topology of the Gromov boundary of this graph for closed surfaces with higher genus. We first prove a bounded geodesic image theorem for the fine curve graph, a consequence of which is the non-compactness of the Gromov boundary. Using this theorem, we are able to show that the Gromov boundary is linearly connected with respect to some visual metric.
\end{abstract}


\vspace{0.25cm}

\setcounter{tocdepth}{2}
\tableofcontents

\input{sections/sec-1.tex}
\input{sections/sec-2.tex}
\input{sections/sec-3.tex}
\input{sections/sec-4.tex}
\input{sections/sec-5.tex}
\input{sections/sec-6.tex}

\printbibliography

\noindent{\sc Yusen Long}\\
\noindent{\sc Université Paris-Saclay, Institut Mathématique d'Orsay, 91405 Orsay, France}\\
\noindent{\sc Université Paris-Est Créteil, CNRS, LAMA UMR8050, F-94010 Creteil, France}

\noindent{\it Email address:} {\tt \href{mailto:yusen.long@universite-paris-saclay.fr}{yusen.long@universite-paris-saclay.fr}} / {\tt \href{mailto:yusen.long@u-pec.fr}{yusen.long@u-pec.fr}}

\vspace{1em}
\noindent{\sc Dong Tan}\\
\noindent{\sc Guangxi University, Guangxi Center for Mathematical Research, 530004 Nanning, China}

\noindent{\it Email address:} {\tt \href{mailto:duzuizhe2013@foxmail.com}{duzuizhe2013@foxmail.com}}
\end{document}

%% file: sections/sec-1.tex
\section{Introduction}
Let $S$ be an orientable connected closed surface. In their work on the diffeomorphism groups of surfaces, Bowden, Hensel, and Webb introduced the concept of the \textit{fine curve graph}, denoted by $\mathcal{C}^{\dagger}(S)$ \cite{bowden2022quasi}. This graph differs from the classical curve graph in that its vertices correspond directly to simple closed essential curves on $S$, rather than the isotopy classes of these curves. An edge in $\mathcal{C}^{\dagger}(S)$ connects two vertices if the corresponding curves are disjoint. This construction enables a finer analysis of the geometry and dynamics associated with diffeomorphisms on $S$.

Indeed, the fine curve graph $\mathcal{C}^\dagger(S)$ shares many similarities with the classical curve graph of $S$ and serves as an analogue of the curve graph for the mapping class group: the graph $\mathcal{C}^\dagger(S)$ is also Gromov hyperbolic \cite{bowden2022quasi}, and the group of homeomorphisms of the surface $S$ is isomorphic to the automorphism group of the fine curve graph \cite{long2021automorphisms}.

Yet compared to the classical curve graph, the fine curve graph $\mathcal{C}^\dagger(S)$ is significantly more complex, and the dynamics of the homeomorphism group acting on $\mathcal{C}^{\dagger}(S)$ are much richer than those of the mapping class group on the curve graph. Notably, $\mathcal{C}^\dagger(S)$ is locally uncountable. Although there are points on the Gromov boundary $\partial\mathcal{C}^\dagger(S)$ that can be represented by certain foliations on the surface \cite{bowden2024boundaryfinecurvegraph}, there also exist homeomorphisms of the surface $S$ that induce parabolic isometries on $\mathcal{C}^\dagger(S)$, which in turn leads to boundary points not representable by foliations \cite{bowden2022quasi, bowden2022rotation}, in contrast to the boundary of the curve graph, where all points are represented by minimal measurable foliations \cite{klarreich2022boundary}.

The geometric and topological interpretation of points at infinity in the fine curve graph remains an active area of research. Current efforts aim to understand the stabilisers of these points through (homological) rotation sets \cite{bowden2022rotation, guiheneuf2023parabolic, guiheneuf2023hyperbolic} and foliations \cite{bowden2024boundaryfinecurvegraph}.

In this paper, our focus is on exploring the topological properties of the Gromov boundary of the fine curve graph. This work represents one of the initial steps in this direction.

For the curve graph, Peter Storm posed an interesting question about the (path)-connectedness of its Gromov boundary, as recorded in \cite[Question 10]{kent2008shadows}. This question was answered affirmatively in \cite{gabai2009almost, leininger2009connectivity, leininger2011universal}. Recently, \cite{wright2023spheres} provided a new proof, showing additionally that the Gromov boundary of the curve graph is linearly connected (see Definition \ref{def_lc}).

For any $o \in \mathcal{C}^\dagger(S)$ and sufficiently small $\varepsilon > 0$, a \textit{visual metric} $\rho_{o,b}$ can be defined on the Gromov boundary, where $b \in (1,1+\varepsilon)$, and this metric is compatible with the topology on the Gromov boundary.

The main goal of this paper is to show the following result of linear connectedness with respect to some visual metric:
\begin{theorem}\label{main_thm}
Let $S$ be an orientable connected closed surface with genus $g\geq 3$, then the Gromov boundary $\partial\mathcal{C}^\dagger(S)$ of the fine curve graph is path connected. 

Moreover, if $o$ is a non-separating curve on $S$, then equipped with the visual metric $\rho_{o,b}$, the Gromov boundary is linearly connected.
\end{theorem}

Following the approach in \cite{wright2023spheres}, this result is obtained by analysing a subgraph that is quasi-isometric to the original Gromov hyperbolic graph. This technique originates in \cite{Bestvina1991TheBO}.

In Section \ref{sec-2}, we first generalise the criterion introduced by Wright in \cite{wright2023spheres} to apply to \textit{any} hyperbolic graph without ``\textit{dead ends}'' (see Proposition \ref{prop_wright}). We then apply Proposition \ref{prop_wright} to a carefully selected subgraph.

In our setting, this subgraph is $\NCo(S)$, the subgraph of $\mathcal{C}^\dagger(S)$ consisting of non-separating curves that are transverse to $o$. This subgraph is quasi-isometric to the subgraph $\mathcal{NC}^\dagger(S)$ of non-separating curves, which in turn is quasi-isometric to the fine curve graph $\mathcal{C}^\dagger(S)$. In addition, the spheres in the subgraph $\mathcal{NC}^\dagger(S)$ are connected.

\begin{theorem}\label{thm1}
Let $S$ be an orientable connected closed surface with genus $g \geq 3$. For any integer $r>0$ and any $o\in\NC^\dagger(S)$, the $r$-sphere $\{\alpha\in\NC^\dagger(S)\colon d^\dagger(\alpha,o)=r\}$ is connected.
\end{theorem}

However, the connectedness of spheres in the curve graph and in its Gromov boundary does not directly help in proving Theorem \ref{main_thm}. Since the fine curve graph is approximated by the \textit{surviving curve graph} on punctured subsurfaces $S-P$ (see Proposition \ref{prop_NC_approx}), the paths connecting two points on the sphere obtained by the method in \cite{wright2023spheres} may involve pant curves on the punctured subsurfaces, which are not essential curves on $S$.

To address the complications introduced by pant curves, we establish a bounded geodesic image theorem for the fine curve graph in Section \ref{sec-4}, following the approach of \cite{webb2013shortproofboundedgeodesic}.
\begin{theorem}[Bounded image theorem]\label{thm_bounded_geodesic_image}
    Given a surface $S$ with $g(S)\geq 2$, there exists a constant $M$ such that whenever $Y$ is a subsurface and $g=\left(\gamma_i\right)$ is a geodesic in $\mathcal{C}^{\dagger}(S)$ such that $\gamma_i$ intersects $Y$ essentially for all $i$, then $\diam_{d_Y^{\dagger}}(\pi_{Y}(g)) \leq M$.
\end{theorem}
Similarly, we also have a version for the non-surviving curve graph (see Theorem \ref{thm_surviving_bounded_geodesic_image}), which is the theorem used to prove the connectedness properties. The primary challenges in the proof arise from the non-transversality between topological curves and the surgeries on the surviving curve graph, both of which are carefully treated in Sections \ref{sec-3} and \ref{sec-4}.

With the bounded geodesic image theorem, we can also conclude that the Gromov boundary of the fine curve graph is not compact, similar to the case of the curve graph.
\begin{theorem}\label{thm_cpt}
Let $S$ be an orientable connected closed surface with genus $g(S)\geq 2$. Then the Gromov boundary $\partial \mathcal{C}^\dagger(S)$ is not compact. 
\end{theorem}

To verify the conditions in Proposition \ref{prop_wright}, the key step is to use the bounded geodesic image theorem to construct paths connecting points on thickened spheres in the subgraph $\NCo(S)$, as desired. These constructions are addressed in Section \ref{sec-5}.

At the end of this paper, we also present several interesting observations concerning the group action of $\homeo(S)$ on the Gromov bordification of $\mathcal{C}^\dagger(S)$.

\subsection*{Acknowledgement}
The first author wishes to thank the hospitality of Guangxi Center for Mathematical Research. The first author is supported by CSC-202108170018 and Tremplin--ERC Starting Grant MAGIC (reference ANR-23-TERC-0007). The second author is partially supported by National Natural Science Foundation of China grant No.12001122, 12271533, 12361014. The authors are grateful to Mladen Bestvina, Bruno Duchesne, Federica Fanoni, Pierre-Antoine Guih\'{e}neuf, Sebastian Hensel, Yongsheng Jia, Roberta Shapiro and Alex Wright for their valuable comments.

%% file: sections/sec-2.tex
\section{Gromov hyperbolic space and bordification}\label{sec-2}

\subsection{Definition.} The notion of $\delta$-hyperbolic space was originally introduced in \cite{gromov1987hyperbolic}. Classical references are \cite{bridson2013metric} and \cite{das2017geometry}, while the latter treats more general cases. Here we cover a very limited part of this topic to clarify some folklore results and introduce the tools that we will later need. For more details, we refer the reader, if interested, to the classical references above.

In the sequel, we will write $O(\delta)$ for an additive error that is at most a multiple of $\delta>0$, {\it i.e.} $f=g+O(\delta)$ if and only if there exists $K>0$ such that $\|f-g\|_\infty <K\delta$.

Let $(X,d_X)$ be a metric space. Let $o\in X$ be an arbitrary base point. Define the function $\langle\cdot,\cdot\rangle_o\colon X\times X\to [0,\infty)$ by
\begin{align}\label{eq2.1}
\langle x,y\rangle_o\coloneqq \frac{1}{2}\big(d_X(x,o)+d_X(y,o)-d_X(x,y)\big)
\end{align}
for every $x,y\in X$. The quantity $\langle x,y\rangle_o$ in (\ref{eq2.1}) is called the {\it Gromov product of $x$ and $y$ based at $o$}, or simply the \textit{Gromov product of $x$ and $y$}.

\begin{definition}[Gromov hyperbolic space]\label{def2.1.3}
A metric space $(X,d_X)$ is {\it Gromov hyperbolic}, or {\it $\delta$-hyperbolic}, if there exists some $\delta\geq 0$ such that
\begin{align}\label{eq2.2}
    \langle x,y\rangle_o\geq \min\left\{\langle x,z\rangle_o,\langle z,y\rangle_o\right\}-\delta
\end{align}
for all $x,y,z,o\in X$. 
\end{definition}

When the metric space is geodesic, \textit{i.e.} a space where every pair of points can be connected by a path isometric to a closed real interval, $\delta$-hyperbolicity is equivalent to the property that every geodesic triangle is $\delta'$-slim, for some uniform $\delta'\geq 0$. This condition is attributed to Rips by Gromov in \cite{gromov1987hyperbolic} and is thus also often called {\it Rips condition}, see \cite[\S 4.3]{das2017geometry} for the exact definition.

A sequence $(x_n)_{n\geq 0}$ in $\delta$-hyperbolic space is {\it Cauchy-Gromov} if $\langle x_n,x_m\rangle_o\to \infty$ as $n,m\to \infty$. We remark that being Cauchy-Gromov of a sequence does not depend on the choice of the base point $o\in X$. A {\it Cauchy-Gromov sequence} $(x_n)_{n\geq 0}$ is {\it equivalent} to another Cauchy-Gromov sequence $(y_m)_{m\geq 0}$, if $\langle x_n,y_m\rangle_o\to \infty$ as $n,m\to \infty$. 

\begin{definition}[{\it Gromov boundary}]
Let $(X,d_X)$ be a Gromov hyperbolic space. Then we define its {\it Gromov boundary} by the equivalence classes of Cauchy-Gromov sequences in $X$ and we denote it by $\partial X$. We say that a sequence $(x_n)_{n\geq 0}$ converges to a boundary point $\xi\in\partial X$ if it is a representative of $\xi$. The disjoint union 
$$
\mathrm{bord}X\coloneqq X\cup \partial X
$$ 
is called the {\it Gromov bordification} of $X$.
\end{definition}

The Gromov product can be extended to the boundary in the following way. For all $x\in X$ and $\xi\in\partial X$, we define
$$\langle x,\xi\rangle_o\coloneqq\inf \left\{\liminf_{n\to \infty}\langle x,x_n\rangle_o\colon(x_n)\in\xi\right\},$$
and also by setting for all $\eta,\xi\in\partial X$
$$\langle \eta,\xi\rangle_o\coloneqq\inf \left\{\liminf_{n,m\to \infty}\langle y_m,x_n\rangle_o\colon(x_n)\in \xi\mathrm{\ and\ }(y_m)\in \eta\right\}.$$
In particular, note that $\langle \xi,\xi\rangle_o=\infty$ for all $\xi \in\partial X$. Moreover, we have the estimation
\begin{align}\label{product_est}
\left|\lim_{n,m\to \infty}\langle x_n, y_m\rangle_o-\langle \xi,\eta\rangle_o\right|=O(\delta)    
\end{align}
for any Cauchy-Gromov sequences $(x_n)\in \xi$ and $(y_n)\in \eta$ and any $o\in X$ \cite[Lemma 3.4.7]{das2017geometry}.

With the Gromov product, the Gromov boundary can be endowed with the uniform structure generated by the  {\it entourages} in the form of
$$\left\{(\xi,\eta)\in \partial X \times \partial X\colon \langle \xi,\eta\rangle_o> R>0\right\},$$
which further makes it into a topological space. Another way to understand the open sets in $\partial X$ is to consider the basis consisting of {\it shadows}, {\it i.e.} subsets given by
\begin{align}\label{eq_shadow}
    \left\{\xi\in\partial X\colon \langle x,\xi\rangle_o>R>0\right\}.
\end{align}
This fact can be easily verified by (\ref{eq2.2}) and (\ref{product_est}) under the setting of Gromov hyperbolicity.

Recall that a map $f\colon (X,d_X)\to (Y,d_Y)$ between two metric spaces is a {\it $(\lambda,k)$-quasi-isometric embedding} if there exist $\lambda\geq 1$ and $k>0$ such that
$$\frac{1}{\lambda}d_X(x,y)-k\leq d_Y\big(f(x),f(y)\big)\leq \lambda d_X(x,y)+k$$
for every $x,y\in X$. If in addition, the map $f$ is {\it essentially surjective}, {\it i.e.} there exists $C>0$ such that $d_Y\big(y,f(X)\big)<C<\infty$ for all $y\in Y$, then $f$ is a quasi-isometry and $X$ and $Y$ are said to be quasi-isometric. We remark that being quasi-isometric is an equivalence relation. A {\it quasi-geodesic} in a metric space $X$ is a quasi-isometric embedding of a real interval into $X$.

The following result is pivotal in large scale geometry {\it {\`a} la Gromov} and is well-known in proper cases. But we remark that it actually holds for non-proper spaces \cite{hasegawa2022gromov,vasala2005gromov}. For convenience, we state it here below. There are several different proofs from different approaches and the one we present here might already be known to \cite{caprace2015amenable}:
\begin{proposition}\label{prop_folklore}
Let $(X,d_X)$ be a $\delta$-hyperbolic geodesic space and $o\in X$ be an arbitrary base point. Then $\partial X$ is in bijection with the equivalence classes of quasi-geodesic rays issued from $o$ and two quasi-geodesic rays are equivalent if they have bounded Hausdorff distance. Moreover, a quasi-isometric embedding $f\colon X\to Y$ induces a topological embedding $f_\partial \colon \partial X\to \partial Y$ and $f_\partial$ is a homeomorphism if $f$ is essentially surjective.
\end{proposition}
\begin{proof}[Sketch of Proof]
First, fix a base point $o\in X$. We remark that for every boundary point $\eta\in\partial X$, we can connect $o$ to $\xi$ by a continuous $(1,N\delta)$-quasi-geodesic for some large $N>0$. Let $\omega$ be a non-principal ultrafilter over $\mathbb{N}$. Take the ultraproduct $u(X)$ of $X$ along $\omega$ in the following way: $u(X)$ is the space consisting of all equivalent {\it bounded} sequences in $X$ and two sequences $(x_n)$ and $(y_n)$ are equivalent if the $\omega$-limit of $d_X(x_n, y_n)$ tends to $0$. The space $u(X)$ is equipped with the limit metric $d_\omega([x],[y])\coloneqq \lim_{n\to \omega}d_X (x_n,y_n)$. It is not hard to see that with this metric, $u(X)$ is geodesic and $\delta$-hyperbolic, where $X\hookrightarrow u(X)$ is embedded isometrically. Moreover, one can verify that some geodesic segments with endpoints on a quasi-geodesic ray or a bi-infinite quasi-geodesic line converge to a geodesic ray or a geodesic line. This fact plays the r{\^o}le of Arzel{\`a}-Ascoli theorem as in the classical proper cases. Now we can conclude the results, including {\it Morse Lemma} for infinite quasi-geodesics, as in the classical cases.
\end{proof}

\begin{remark}
In \cite[\S 7]{gromov1987hyperbolic}, a Gromov hyperbolic space $X$ is {\it ultracomplete} if for any two points $x,y\in\mathrm{bord} X$, there exists a geodesic connecting $x$ to $y$. The proof above actually shows that $u(X)$ is ultracomplete. If a Gromov hyperbolic geodesic space is proper, then it is ultracomplete. But the converse is not true: the {\it curve graph} of a compact hyperbolic surface is ultracomplete but not proper \cite[Lemma 5.14]{minsky2010classification}.
\end{remark}

\subsection{Visual metric and boundary connectedness.} Another powerful tool to study the Gromov boundary is the following notion of visual metric that is carefully treated in \cite[\S 3.6]{das2017geometry} and \cite{vasala2005gromov} under general settings:
\begin{theorem}[V{\"a}s{\"a}il{\"a}]
Let $(X,d_X)$ be a Gromov hyperbolic space. Then $\partial X$ is completely metrizable. For any based point $o\in X$, the topology is compatible to the {\it visual metric} $\rho_{o,b}$ from $o\in X$ with parameter $b\in(1,1+\varepsilon)$ for some small enough $\varepsilon>0$. Moreover, we have the following estimate
\begin{align}\label{est_visual_metric}
b^{-\langle \xi,\eta\rangle_o}/4\leq \rho_{o,b}(\xi,\eta)\leq b^{-\langle \xi,\eta\rangle_o}
\end{align}
for every $\xi,\eta\in \partial X$.
\end{theorem}

Let $X$ be a $\delta$-hyperbolic geodesic space and let $\gamma\subset X$ be a geodesic. We define for a point $x\in X$ its {\it nearest points projection} $\pi_\gamma(x)$ by the points $y\in\gamma$ satisfying $d_X(x,y)=d_X(x,\gamma)$. We remark that $\pi_\gamma(x)$ is not empty in a geodesic space: it consists of the minima of the continuous proper function $d_X(x,\cdot)\colon \gamma \to [0,\infty)$.

The first result about this notion is referred to as the {\it reverse triangle inequality}:
\begin{proposition}[\cite{maher2018random}, Proposition 2.2]\label{prop_proj_geo}
Let $X$ be $\delta$-hyperbolic and $\gamma$ be a geodesic. For any $x\in X$ and any $p\in\pi_\gamma(x)$, we have
\begin{align}\label{eq_2}
    d(x,y)=d(x,p)+d(p,y)+O(\delta)
\end{align}
for all $y\in \gamma$.
\end{proposition}
Clearly from (\ref{eq_2}), the nearest point projection is coarsely well-defined: the diameter of $\pi_\gamma(X)$ is bounded uniformly by a multiple of $\delta$.

Using this inequality, we can deduce the following estimate, which essentially translates the thinness of geodesic triangles in Gromov hyperbolic spaces into the language of nearest point projection:
\begin{corollary}
Let $X$ be a geodesic $\delta$-hyperbolic space. Let $x,y,o\in X$ be three points and $\gamma\coloneqq[o,y]$ be a geodesic segment. Then for any $p\in \pi_\gamma(x)$,
\begin{align}\label{eq_3}
    \langle x,y\rangle_o=d_X(o,p)+O(\delta).
\end{align}
\end{corollary}
\begin{proof}
Applying (\ref{eq_2}) for $d_X(x,o)$ and $d_X(x,y)$ appeared in (\ref{eq2.1}) yields the result.
\end{proof}

Now we can prove the following lemma about the distance between two boundary points under a visual metric:
\begin{lemma}\label{lem_vm_est_qg}
Let $X$ be a $\delta$-hyperbolic geodesic space. Fix a base point $o\in X$. Let $\rho_{o,b}$ be a visual metric from $o$ with parameter $b\in(1,1+\varepsilon)$ for sufficiently small $\varepsilon>0$. Then for every large enough $E\gg 1$, $\lambda\geq 1$ and $k\geq 0$, there exists a constant $C\coloneqq C(b,k,E,\delta)>0$ such that for every two distinct points $\xi,\eta\in\partial \Gamma$ and for any $\ell_1,\ell_2$ continuous $(\lambda,k)$-quasi-geodesics connecting $o$ to $\xi$ and $\eta$ respectively, if
$$r\coloneqq \sup\{t\geq 0\colon d_X(\ell_1(t),\ell_2(t))\leq E\},$$
then 
\begin{align}\label{vm_est_qg}
b^{-\lambda r}/C\leq \rho_{o,b}(\xi,\eta)\leq Cb^{-r/\lambda}.
\end{align}
\end{lemma}
\begin{proof}
By continuity of $\ell_1$ and $\ell_2$, if $x\coloneqq\ell_1(r)$ and $y\coloneqq \ell_2(r)$, then $d_X(x,y)=E$. Let $\gamma\coloneqq[o,y]$ be a geodesic segment and $p\in \pi_\gamma(x)$. Applying the reverse triangle inequality (\ref{eq_2}) twice, we have
$$d_X(p,y)=d_X(x,y)-d_X(x,p)+O(\delta)=E-d_X(x,o)+d_X(o,p)+O(\delta).$$
Using the fact that $d_X(o,y)=d_X(o,p)+d_X(p,y)$, we can further deduce that
$$d_X(o,p)=\frac{1}{2}\left(d_X(x,o)+d_X(y,o)-E\right)+O(\delta).$$
Using the $(\lambda,k)$-quasi-geodesicity, we have the inequality
\begin{align}\label{ineq_1}
    \frac{r}{\lambda}-k-E/2+O(\delta)\leq d_X(o,p)\leq\lambda r+k-E/2 +O(\delta).
\end{align}
But by (\ref{product_est}), when $E$ is sufficiently large, we have $\langle x,y\rangle_o=\langle \xi,\eta\rangle_o+O(\delta)$. Combining this fact with (\ref{eq_3}) and (\ref{ineq_1}), we can deduce that
$$\frac{r}{\lambda}-k-E/2+O(\delta)\leq \langle \xi,\eta\rangle_o\leq\lambda r+k-E/2 +O(\delta).$$
By setting $C\coloneqq 4b^{k+E/2+O(\delta)}$, we can deduce the desired inequality via applying the estimate (\ref{est_visual_metric}).
\end{proof}
\begin{remark}
As mentioned in the proof of Proposition \ref{prop_folklore}, in a geodesic $\delta$-hyperbolic space, each boundary point $\xi \in \partial X$ can be connected to the base point $o$ by a continuous $(1,N\delta)$-quasi-geodesic. So the $\lambda$ from Lemma \ref{lem_vm_est_qg} can be taken as $1$.
\end{remark}

At this moment, we introduce the following notion:
\begin{definition}[Linear connectedness]\label{def_lc}
A metric space $(X,d_X)$ is {\it linearly connected} if there is a constant $L> 0$ such that for each pair $x,y\in X$ there is a compact connected set $K\subset X$ containing $x,y$ with diameter at most $Ld_X(x,y)$.
\end{definition}

The application of Lemma \ref{lem_vm_est_qg} is the following proposition inspired by \cite[Proof of Proposition 5.2]{mackey2020quasi} and is an improvement of \cite[Proposition 2.2]{wright2023spheres}.

For notations, let $S_o(r)$ be the sphere centred at $o$ of radius $r>0$ and $B_o(r)$ be the ball centred at $o$ of radius $r>0$.

\begin{proposition}\label{prop_wright}
Let $\Gamma$ be a Gromov hyperbolic simplicial graph and let $o \in \Gamma$ be a base point. Suppose that the following holds.
\begin{enumerate}[label=(G\arabic*)]
    \item Every vertex of $\Gamma$ is adjacent to point of $\Gamma$ that is $1$ farther from $o$. \label{G1}
\end{enumerate}
If in addition there is some $D>0$ such that for all $r \geq0$, the following conditions are satisfied:
\begin{enumerate}[label=(G\arabic*)]
    \setcounter{enumi}{1}
    \item For every $z \in S_o(r)$ and $x, y \in S_{o}(r+1) \cap B_z(1)$ there exists a path
        $$x=x_0,\ x_1,\ \dots,\ x_{\ell}=y$$
        with
        $$x_i \in\left(\Gamma-B_o(r)\right) \cap B_x(D)$$
        for $0 \leq i \leq \ell$. \label{G2}
    \item For every adjacent pair $x, y \in S_o(r)$ there exists a path
        $$x=x_0,\ x_1,\ \dots,\ x_{\ell}=y$$
        with
        $$x_i \in\left(\Gamma-B_o(r)\right) \cap B_x(D)$$
    for $0<i<\ell$. \label{G3}
\end{enumerate}
Then equipped with the visual metric $\rho_{o,b}$ from $o$ with parameter $b\in (1,1+\varepsilon)$ for some small enough $\varepsilon>0$, the Gromov boundary $\partial\Gamma$ is linearly connected and path connected.
\end{proposition}
\begin{proof}
The proof is a straightforward adaptation of the proof of \cite[Proposition 2.2]{wright2023spheres}. Let $\xi_0,\xi_1\in\partial X$ be any two points on the Gromov boundary. Now connect $\xi_0$ and $\xi_1$ to $o$ respectively by continuous $(1,k)$-quasi-geodesic rays $\gamma_0$ and $\gamma_1$. By the same arguments in \cite[Proof of Proposition 2.2]{wright2023spheres}, \ref{G1} ensures the existence of a dense subset $I\subset [0,1]$ such that for all $t\in I$, there is a geodesic ray $\gamma_t\subset \Gamma$ connecting $o$ to a point $c(t)\in\partial \Gamma$. Since the family $\left(\gamma_t\right)_{t\in I\cup\{0,1\}}$ consists of continuous $(1,k)$-quasi-geodesic rays, by Lemma \ref{lem_vm_est_qg}, we can deduce from \ref{G2} and \ref{G3} the same estimation as in \cite[Proof of Proposition 2.2]{wright2023spheres}, {\it i.e.} there exists $L>0$ independent of $\xi_0$ and $\xi_1$ such that the diameter $\mathrm{diam}\big(c(I)\big)<L\rho_{o,b}(\xi_0,\xi_1)$ and that $c\colon I\to \partial \Gamma$ is uniformly continuous. The desired path is then the closure of $c(I)$.
\end{proof}
\begin{remark}
Proposition \ref{prop_wright} cannot be obtained by applying \cite[Proposition 2.2]{wright2023spheres}.
\end{remark}

%% file: sections/sec-3.tex
\section{Fine curve graph and its subgraphs}\label{sec-3}

\subsection{Fine curve graph.} Let $S$ be an orientable connected closed surface of finite type with genus $g \geq 2$. Such a surface admits a complete hyperbolic metric. In a recent paper, Bowden, Hensel, and Webb introduced the {\it fine curve graph} of $S$ \cite{bowden2022quasi}, a combinatorial tool designed to study the homeomorphism group $\homeo(S)$ and the diffeomorphism group $\mathrm{Diff}(S)$ of $S$. This construction serves as an analogue to the well-known curve graph.

An {\it essential simple closed smooth curve} $\gamma$ on $S$ is a smooth embedding of the circle, $\gamma: S^1 \hookrightarrow S$, with the property that $\gamma$ does not bound a disc. Unless stated otherwise, all curves considered on $S$ will be assumed to be essential, simple, closed, and smooth.

The {\it fine curve graph} $\mathcal{C}^\dagger(S)$ is a graph whose vertices correspond to essential simple closed smooth curves on $S$. Two vertices are connected by an edge if the corresponding curves are disjoint. Equipped with the simplicial metric, the fine curve graph is Gromov hyperbolic \cite{bowden2022quasi}.

\begin{remark}\label{rem_topological_version}
It is possible to construct the fine curve graph using non-differentiable curves, {\it i.e.}, $\mathcal{C}^0$-embeddings of $S^1 \hookrightarrow S$. This construction is referred to as the {\it topological version of the fine curve graph}. Notably, any topological curve can be isotoped to a differentiable curve that remains disjoint from it. This can be achieved by isotopically smoothing (or "mollifying") the curve and then moving it slightly. Consequently, there exists a $1$-Lipschitz embedding of the fine curve graph of differentiable curves into the fine curve graph of non-differentiable curves, which is also $1$-dense. Thus, the two versions of the fine curve graph are quasi-isometric, meaning they are equivalent from the perspective of large-scale geometry.
\end{remark}

Next, we examine the relationship between the fine curve graph and the classical curve graph. For any orientable surface, including the torus, there is a natural projection from the fine curve graph to the classical curve graph.

\begin{proposition}\label{prop_1}
Let $S$ be as above. Then the map $[\cdot] \colon (\mathcal{C}^\dagger(S), d^\dagger) \to (\mathcal{C}(S), d_S)$, which sends a simple closed curve to its isotopy class, is $1$-Lipschitz.
\end{proposition}

\begin{proof}
Let $a, b \in \mathcal{C}^\dagger(S)$ be any two distinct simple closed curves on $S$. Consider a geodesic $a = x_0, x_1, \dots, x_n = b$ in $\mathcal{C}^\dagger(S)$ connecting $a$ to $b$. For each $i = 0, \dots, n - 1$, the curves $x_i$ and $x_{i+1}$ are disjoint. Moreover, either $x_i$ is isotopic to $x_{i+1}$, or $[x_i] \neq [x_{i+1}]$. 

In the latter case, the isotopy classes $[x_i]$ and $[x_{i+1}]$ have disjoint representatives, specifically the curves $x_i$ and $x_{i+1}$. This implies that $[x_i]$ and $[x_{i+1}]$ are adjacent in $\mathcal{C}(S)$. Consequently, we have
\[
d_S\big([x_i], [x_{i+1}]\big) \leq d^\dagger(x_i, x_{i+1})
\]
in each case. Therefore,
\[
d_S\big([a], [b]\big) \leq \sum_{i=0}^{n-1} d_S\big([x_i], [x_{i+1}]\big) \leq \sum_{i=0}^{n-1} d^\dagger(x_i, x_{i+1}) = n = d^\dagger(a, b).
\]
This shows that the map $[\cdot] \colon \mathcal{C}^\dagger(S) \to \mathcal{C}(S)$ is $1$-Lipschitz.
\end{proof}

Suppose now that $S$ is equipped with a hyperbolic metric. Then each simple closed curve on $S$ is isotopic to a unique geodesic representative \cite[Proposition 1.3]{farb2011primer}. Thus, the map $\gamma \colon \mathcal{C}(S) \hookrightarrow \mathcal{C}^\dagger(S)$, defined by sending an isotopy class $a \in \mathcal{C}(S)$ to its geodesic representative $\gamma(a) \in \mathcal{C}^\dagger(S)$, is well-defined. A similar argument to that in Proposition \ref{prop_1} shows the following result.

\begin{proposition}\label{prop_isom}
Let $S$ be a hyperbolic surface. Then the map $\gamma \colon \mathcal{C}(S) \hookrightarrow \mathcal{C}^\dagger(S)$, which assigns to each isotopy class $a \in \mathcal{C}(S)$ its unique geodesic representative $\gamma(a) \in \mathcal{C}^\dagger(S)$, is an isometric embedding.
\end{proposition}

\begin{proof}
Let $\alpha, \beta \in \mathcal{C}(S)$, and let $\alpha = y_0, \dots, y_m = \beta$ be a geodesic in $\mathcal{C}(S)$ connecting $\alpha$ to $\beta$. For each $i = 0, \dots, m - 1$, the isotopy classes $y_i$ and $y_{i+1}$ admit disjoint representatives. Since geodesics are in minimal position \cite[Corollary 1.9]{farb2011primer}, it follows that the geodesic representatives $\gamma(y_i)$ and $\gamma(y_{i+1})$ must also be disjoint. Thus, we have $d^\dagger\big(\gamma(y_i), \gamma(y_{i+1})\big) = 1$ for each $i$.

As a result,
\[
d^\dagger\big(\gamma(\alpha), \gamma(\beta)\big) \leq \sum_{i=0}^{m-1} d^\dagger\big(\gamma(y_i), \gamma(y_{i+1})\big) = m = d_S(\alpha, \beta),
\]
showing that $\gamma$ is $1$-Lipschitz. 

Now, note that the composition $[\cdot] \circ \gamma = \mathrm{Id} \colon \mathcal{C}(S) \to \mathcal{C}(S)$ is the identity map. By Proposition \ref{prop_1}, we have
\[
d_S(\alpha, \beta) = d_S\big([\gamma(\alpha)], [\gamma(\beta)]\big) \leq d^\dagger\big(\gamma(\alpha), \gamma(\beta)\big).
\]
Therefore, $d_S(\alpha, \beta) = d^\dagger\big(\gamma(\alpha), \gamma(\beta)\big)$, which shows that $\gamma$ is an isometry. 

Finally, it is clear that $\gamma$ is also an embedding, completing the proof.
\end{proof}

\subsection{Distance estimation.} Similar to the curve graph, we can estimate distances between certain curves by examining the cardinality of their intersections. To proceed, we introduce the following concept:

\begin{definition}[Transversality]
Let $\alpha, \beta$ be two simple closed curves on $S$. We say that they are {\it transverse} if exactly one of the following holds:
\begin{itemize}
    \item $\alpha \cap \beta = \emptyset$;
    \item For each point $p \in \alpha \cap \beta$, there exists an open neighbourhood $N \subset S$ of $p$ and a diffeomorphism $\varphi \colon N \to D \subset \mathbb{R}^2$ onto the unit disc $D$, such that $\varphi(\alpha \cap N) = (-1,1) \times \{0\}$ and $\varphi(\beta \cap N) = \{0\} \times (-1,1)$.
\end{itemize}
\end{definition}

Note that if two curves are transverse, then by the compactness of the surface, they intersect at only finitely many points.

A {\it bigon} between two transverse curves $\alpha$ and $\beta$ on a surface is a connected component of the complement $S \setminus (\alpha \cup \beta)$ that is homeomorphic to a disc and contains exactly one subarc of $\alpha$ and one subarc of $\beta$. Two transverse curves $\alpha$ and $\beta$ are said to be in {\it minimal position} if they contain no bigon \cite[Proposition 1.7]{farb2011primer}.

One advantage of smoothness is that curves can be taken as geodesics under an appropriate hyperbolic metric:

\begin{proposition}[Proposition 10, Exposé 3, \cite{fathi1991travaux}]\label{prop_metric}
Let $S$ be a connected orientable surface of genus $g \geq 2$. Let $\gamma$ be a curve on $S$. Then there exists a hyperbolic metric on $S$ such that $\gamma$ is geodesic. Moreover, if $\gamma$ and $\gamma'$ are two smooth curves in minimal position, then there exists a hyperbolic metric on $S$ such that both $\gamma$ and $\gamma'$ are geodesic.
\end{proposition}

In particular, by combining Proposition \ref{prop_metric} and Proposition \ref{prop_isom}, we obtain the following result:

\begin{corollary}\label{lem_lower_dist}
    For any two vertices $\alpha, \beta \in \mathcal{C}^{\dagger}(S)$ that are in minimal position on $S$, we have $d_{S}([\alpha], [\beta]) = d^{\dagger}(\alpha, \beta)$, where $d_{S}$ denotes the simplicial distance in the curve graph $\mathcal{C}(S)$.
\end{corollary}

Now, let $P \subset S$ be a finite collection of points. Adapting terminology from \cite{bowden2022quasi}, recall that a curve $\gamma \in \mathcal{C}^\dagger(S - P)$ is called {\it surviving} if it remains essential in $S$. We denote by $\mathcal{C}^s(S - P)$ the surviving curve graph of the punctured surface $S - P$. An important property in the metric geometry of the fine curve graph is that it can be approximated by the surviving curve graphs of these punctured subsurfaces:

\begin{proposition}[\cite{bowden2022quasi}]\label{prop_NC_approx}
Suppose that $\alpha,\beta \in \mathcal{C}^{\dagger}(S)$ are transverse, and that $\alpha$ and $\beta$ are in minimal position in $S-P$, where $P\subset S$ is finite and disjoint from $\alpha\cup \beta$. Then 
$$d^{\dagger}(\alpha, \beta)=d_{\mathcal{C}^s(S-P)}\left([\alpha]_{S-P},[\beta]_{S-P}\right)$$
where $[\alpha]_{S-P}$ and $[\beta]_{S-P}$ are respectively the isotopy classes of $\alpha$ and $\beta$ in $S-P$. Moreover, if $\alpha$ and $\beta$ are two non-separating curves, then 
$$d^\dagger(\alpha,\beta)=d_{\mathcal{NC}(S-P)}([\alpha]_{S-P},[\beta]_{S-P}),$$
and that the inclusion map $\mathcal{NC}^\dagger(S)\hookrightarrow\mathcal{C}^\dagger(S)$ is a $(1,2)$-quasi-isometry.
\end{proposition}

Similar to \cite[Lemma 2.1]{masur1999geometry}, we can deduce the following distance estimation in the fine curve graph:
\begin{proposition}\label{lem_upper_dist}
    Let $S$ be a surface with genus $g(S)\geq 2$. For any two transverse curves $\alpha, \beta \in \mathcal{C}^\dagger(S)$, we have $d^\dagger(\alpha, \beta) \leq 2  |\alpha\cap \beta|+1$.
\end{proposition}
\begin{proof}
If $|\alpha \cap \beta| = 0$, then the formula is trivially true by the definition of the simplicial distance in $\mathcal{C}^\dagger(S)$.

If $|\alpha \cap \beta| = 1$, then there is no bigon between $\alpha$ and $\beta$, meaning that these two curves are in minimal position. Since $g(S) \geq 2$, the subsurface $S - (\alpha \cap \beta)$ has positive genus, which implies that we can find a non-separating (and thus surviving) curve disjoint from both $\alpha$ and $\beta$, so that $d^\dagger(\alpha, \beta) = 2$.

Now suppose $|\alpha \cap \beta| \geq 2$. Let $P \subset S$ be a finite set of points such that $\alpha$ and $\beta$ are in minimal position on $S - P$. We can perform the same surgery on $S - P$ for $\beta$ as in \cite[Lemma 2.1]{masur1999geometry}. Note that we will have two possibilities, $\beta_1$ and $\beta_2$, obtained by replacing a subarc of $\beta$ with a segment isotopic to a subarc of $\alpha$. Both $\beta_1$ and $\beta_2$ remain essential on $S - P$ by \cite[Lemma 2.1]{masur1999geometry}. It suffices to show that the resulting curve is a surviving curve. If both $\beta_1$ and $\beta_2$ are non-surviving, then so will $\beta$, which contradicts the assumption that $\beta \in \mathcal{C}^\dagger(S)$ and survives on $S - P$. 

By induction, we can then find a path of length $2 |\alpha \cap \beta| + 1$ connecting $[\alpha]_{S - P}$ to $[\beta]_{S - P}$ in the surviving curve graph $\mathcal{C}^s(S - P)$, which leads to the desired distance estimate in $\mathcal{C}^\dagger(S)$ after applying Proposition \ref{prop_NC_approx}.
\end{proof}

Recall that a {\it regular neighbourhood} $\mathcal{N}$ of a simple closed curve $\gamma$ in $S$ is a locally flat, compact submanifold of $S$ of codimension $0$, which is a topological neighbourhood of $\gamma$ such that the inclusion $\gamma \subset \mathcal{N}$ is a simple homotopy equivalence, and $\gamma$ is a strong deformation retract of $\mathcal{N}$. We note that, by definition, this regular neighbourhood can be chosen arbitrarily small. Furthermore, if the surface is endowed with a metric, then any annular neighbourhood with sufficiently small width will be a regular neighbourhood.

\begin{lemma}\label{lem_pert}
Let $\alpha \in \mathcal{C}^\dagger(S)$ and let $F \subset \mathcal{C}^\dagger(S)$ be a finite collection of pairwise transverse curves. Then there exists a curve $\alpha' \in \mathcal{C}^\dagger(S)$, isotopic to $\alpha$, that is disjoint from $\alpha$ and transverse to all curves in $F$.
\end{lemma}
\begin{proof}
Define 
\[
C \coloneqq \{x \in S : x \in \gamma_i \cap \gamma_j \text{ for some } \gamma_i, \gamma_j \in F \text{ and } i \neq j\}.
\]
Let $N(\alpha)$ be a sufficiently small regular neighbourhood of $\alpha$. Consider a small enough regular neighbourhood $N(C)$ of $C$ and a sufficiently small neighbourhood $N_i$ of each $\gamma_i$, such that $N_i \cap N_j \subset N(C)$ whenever $i \neq j$. First, isotope $\alpha$ to a curve $\alpha_0$ away from $N(C) \cup N(\alpha)$. Note that any isotopy supported within $N_i - N(C)$ will not affect the intersection behaviour between $\alpha_0$ and $\gamma_j$ for all $j \neq i$. Hence, we can isotope $\alpha_0$ within $N_i - \big(N(C) \cup N(\alpha)\big)$ for $i = 1, \dots, |F|$. In this way, we can find a curve $\alpha'$ isotopic to $\alpha$, disjoint from $\alpha$, and transverse to all $\gamma_i \in F$.
\end{proof}

\begin{remark}\label{rem_perturb}
    We note that the perturbation in Lemma \ref{lem_pert} can be performed within an arbitrarily small regular neighbourhood. All \textit{isotopy perturbations} in the sequel will follow the same basic method, although the specific situation might vary slightly.
\end{remark}

\subsection{Non-separating subgraph.} The {\it non-separating fine curve graph} $\NC^\dagger(S)$ is the subgraph of $\mathcal{C}^\dagger(S)$ consisting of vertices corresponding to non-separating curves. When equipped with the simplicial metric, these two connected graphs are quasi-isometric, with the inclusion map being a $(1,1)$-quasi-isometry, and both graphs are simultaneously Gromov hyperbolic. Hence, by Proposition \ref{prop_folklore}, the Gromov boundary $\partial \NC^\dagger(S)$ is homeomorphic to $\partial \mathcal{C}^\dagger(S)$.

To make full use of the known results about the classical non-separating curve graph, let us fix an arbitrary base point $o \in \mathcal{NC}^\dagger(S)$ and denote by $\NCo(S)$ the subgraph of $\mathcal{NC}^\dagger(S)$ consisting of vertices that are all transverse to $o$. This subgraph inherits the subgraph distance $d_o^\dagger$, but not the induced distance $d^\dagger$ from $\NC^\dagger(S)$.

\begin{proposition}\label{prop_incl}
The inclusion map $i: \NCo(S)\hookrightarrow\NC^\dagger(S)$ is a $(1,2)$-quasi-isometry, hence $\partial \NCo(S)\simeq \partial \NC^\dagger(S)$.
\end{proposition}
\begin{proof}
We first show that $\NCo(S)$ is $1$-dense in $\NC^\dagger(S)$. Consider any $\alpha\in\NC^\dagger(S)$. If $\alpha$ is not transverse to $o$, then, by Lemma \ref{lem_pert}, there exists another curve $\alpha^{\prime} \in \mathcal{NC}^{\dagger}(S)$ that is disjoint from $\alpha$ and is transverse to $o$. Consequently, we have
$$
d^{\dagger}\left(\alpha, \alpha^{\prime}\right) = d^{\dagger}\left(\alpha,\NCo(S)\right) = 1.
$$
This shows that the inclusion map is essentially surjective. 

If $\alpha,\beta\in\NCo(S)$ are transverse, then by taking a finite subset $P\subset S$ such that $o,\alpha,\beta$ are pairwise in minimal position on $S-P$, Proposition \ref{prop_NC_approx} indicates that
$$d^\dagger_o(\alpha,\beta)=d^\dagger(\alpha,\beta)=d_{\mathcal{N} \mathcal{C}(S-P)}\left([\alpha]_{S-P}, [\beta]_{S-P}\right).$$
Now for an arbitrary pair $\alpha,\beta\in\NCo(S)$, we can perturb them to a transverse pair, which will cause an additive error at most $2$. This implies that the inclusion map is a $(1,2)$-quasi-isometry. As a result, their Gromov boundaries are homeomorphic (see Proposition \ref{prop_folklore}).
\end{proof}

Finally, we need to point out the following fact about the comparison between two distances $d^\dagger$ and $d^\dagger_o$:
\begin{lemma}\label{radial_dist_pres}
    For any $\alpha \in \NCo(S)$, we always have $d_{o}^{\dagger}(o, \alpha) = d^{\dagger}(o, \alpha)$.
\end{lemma}
\begin{proof}
Indeed, for any $\alpha \in \NCo(S)$, it is clear that the subgraph distance is always bounded from above by the induced distance, {\it i.e.} $d^{\dagger}(\alpha, o) \leq d_o^{\dagger}(\alpha, o)$. Conversely, because $\alpha$ and $o$ are transverse, by Proposition \ref{prop_NC_approx}, there exists a finite subset $P \subset S$ disjoint from $o \cup \alpha $ such that $o$ and $\alpha$ are in minimal position in $S-P$. Moreover, we have
$$d_{\NC(S-P )}([o]_{S-P}, [\alpha]_{S-P} ) =d^\dagger(o,\alpha).$$
Let $[o]_{S-P}=[\nu_0]_{S-P},\ [\nu_1]_{S-P},\ \dots,\ [\nu_n]_{S-P}=[\alpha]_{S-P}$ be a geodesic in $\NC(S-P)$. By Proposition \ref{prop_metric}, we can choose a hyperbolic metric on $S-P$ so that $o$ and $\alpha$ are simultaneously geodesic. Let $\nu_i$ be the geodesic representative of $[\nu_i]_{S_P}$. Now we can find a path
$$
o = \nu_0,\nu_1, \cdots, \nu_n = \alpha,
$$
connecting $o$ to $\alpha$ in $\NCo(S)$. This implies that $d_{o}^{\dagger}(o,\alpha) \leq n = d^\dagger(o,\alpha)$. This proves the desired result.
\end{proof}

\begin{remark}\label{rem_radii_non_sep}
Using the $d_{\mathcal{C}^s(S-P)}$-distance estimation for curve in minimal position instead of $d_{\mathcal{NC}(S-P)}$, similar arguments in Lemma \ref{radial_dist_pres} also apply to non-separating curves, {\it i.e.} for any curve $\alpha\in \mathcal{C}^\dagger(S)$ transverse to $o$ with $d^\dagger(o,\alpha)=r$, we can find a geodesic path of length $r$ in $\mathcal{C}^\dagger(S)$ connecting $o$ to $\alpha$ such that all vertices on the path are curves transverse to $o$.
\end{remark}

%% file: sections/sec-4.tex
\section{Bounded geodesic image theorem}\label{sec-4}

\subsection{Subsurface projection.}\label{subsec_s_p} Let $S$ be an orientable, connected closed surface of genus $g \geq 2$. Let $Y$ be a compact, proper, {\it non-sporadic} subsurface of $S$. Assume that $\partial Y$ is a finite collection of smooth curves in $\mathcal{C}^\dagger(S)$.

Let $\alpha \in \mathcal{C}^\dagger(S)$ be a simple closed curve on $S$, and let $Y$ be a subsurface as defined above. We say that $\alpha$ intersects $Y$ {\it essentially} if at least one connected component of $\alpha \cap Y$ is an essential arc or curve in $Y$, {\it i.e.}, not isotopically trivial in $Y$. We remark that a general curve $\alpha$ may not intersect $\partial Y$ transversely and this notion of essential intersection may not be invariant under isotopy.

If $\alpha$ intersects $Y$ essentially, then fix an essential arc $\tau\subset \alpha\cap Y$ and denote by $c_\pm\subset \partial Y$ the two boundary components where the endpoints of $\tau$ are located. Note that $c_+$ and $c_-$ might coincide. Let $\mathcal{N}$ be a regular neighbourhood of $\tau\cup c_\pm$ in $Y$ with a non-trivial smooth boundary, {\it i.e.} a boundary that is an essential curve in $Y$.

\begin{figure}[!ht]
    \centering
    \includegraphics[width=0.6\linewidth]{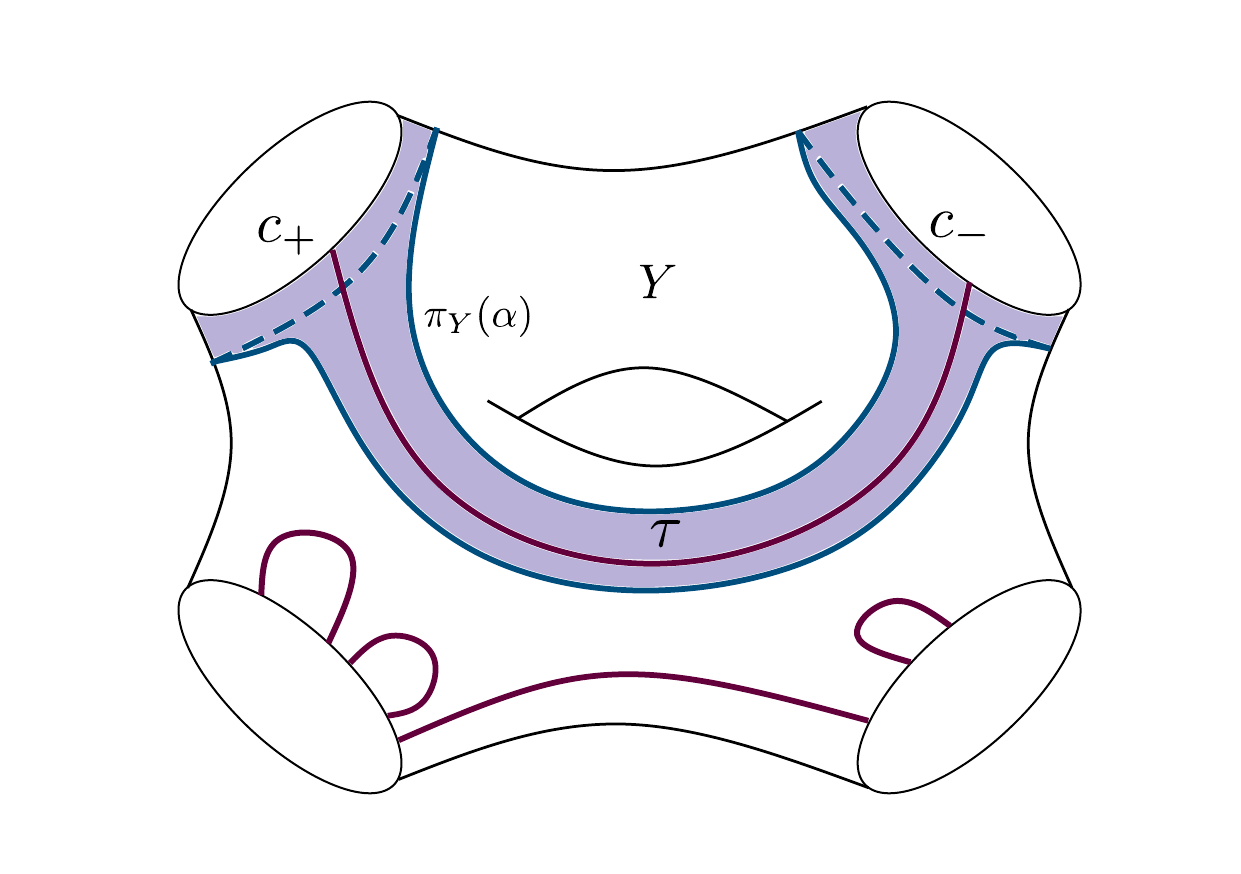}
    \caption{A possible subsurface projection $\pi_Y(\alpha)$ of curve $\alpha$ on subsurface $Y$.}
    \label{fig 1}
\end{figure}

\begin{definition}[Subsurface projection]\label{def_subsur_proj}
Suppose that $S$ and $Y$ are given as above. Let $\mathcal{C}^\dagger(Y)$ be the fine curve graph for $Y$ and let $\mathcal{P}\left(\mathcal{C}^\dagger(Y)\right)$ be its power set. We define a map $\pi_Y\colon \mathcal{C}^\dagger(S)\to\mathcal{P}\left(\mathcal{C}^\dagger(Y)\right)$ in the following way: for each $\alpha\in \mathcal{C}^\dagger(S)$, the image $\pi_Y(\alpha)$ is defined as
\begin{itemize}[topsep=0pt, itemsep=-1ex, partopsep=1ex, parsep=1ex ]
    \item $\{\alpha\}$ if $\alpha\subset Y$;
    \item $\emptyset$ if $\alpha$ does not intersect $Y$ essentially;
    \item the union of all $\partial \mathcal{N}\cap Y$ if $a\not\subset Y$ but $\alpha$ intersects $Y$ essentially, where $\mathcal{N}$ is a regular neighbourhood as above.
\end{itemize}
The map $\pi_Y$ is called the {\it subsurface projection} of $\alpha$ on $Y$.
\end{definition}

Nevertheless, this map is coarsely well-defined:

\begin{proposition}\label{prop_ssurface_proj_well_defined}
With the simplicial metric $\big(\mathcal{C}^\dagger(Y),d^\dagger_Y\big)$, for any curve $\alpha \in \mathcal{C}^\dagger(S)$, the diameter of the set $\pi_Y(\alpha)$ is bounded in $\mathcal{C}^\dagger(Y)$, {\it i.e.}, 
$$\diam_{d^\dagger_Y}(\pi_Y(\alpha)) \leq 11.$$
\end{proposition}

\begin{proof}
It suffices to prove the case where $\alpha$ intersects $Y$ essentially but is not contained in $Y$, as other situations are trivial. Let $\tau_1$ and $\tau_2$ be two essential arcs in $\alpha \cap Y$. We note that $\tau_1$ and $\tau_2$ are either identical or disjoint. Denote by $c^1_\pm$ and $c^2_\pm$ the boundary components of $Y$ where the starting and ending points of $\tau_1$ and $\tau_2$ are located, respectively. For $i = 1, 2$, let $\gamma_i$ be a boundary component of a regular neighbourhood $\mathcal{N}_i$ of $\tau_i \cup c^i_{\pm}$.

We can choose two smaller regular neighbourhoods $\widehat{\mathcal{N}_i}$ of $\tau_i \cup c^i_\pm$ for $i = 1, 2$, with smooth boundary components within $Y$, such that $\widehat{\mathcal{N}_i} \subset \mathcal{N}_i$ and $\partial \widehat{\mathcal{N}_i}$ and $\partial \widehat{\mathcal{N}_j}$ are in minimal position for $i \neq j$. Let $\widehat{\gamma_i}$ denote a boundary component of $\partial \widehat{\mathcal{N}_i}$. Due to disjointness, we have $d_{Y}^{\dagger}(\widehat{\gamma_i}, \gamma_i) = 1$ for $i = 1, 2$. 

Observe that $|\widehat{\gamma_1} \cap \widehat{\gamma_2}| \leq 4$. Since $\widehat{\gamma_1}$ and $\widehat{\gamma_2}$ are in minimal position within $Y$, we apply Proposition \ref{lem_upper_dist} to obtain
$$
d_{Y}^{\dagger}(\gamma_1, \gamma_2) \leq d_{Y}^{\dagger}(\gamma_1, \widehat{\gamma_1}) + d_Y^\dagger(\widehat{\gamma_1}, \widehat{\gamma_2}) + d_{Y}^{\dagger}(\gamma_2, \widehat{\gamma_2}) \leq 11.
$$
This completes the proof.
\end{proof}

In addition to Proposition \ref{prop_ssurface_proj_well_defined}, we have several other distance estimates for curves that intersect the subsurface essentially and their corresponding subsurface projections.

\begin{proposition}\label{prop_disjoint_subsurface_proj}
    Let $S$, $Y$, $d^\dagger_Y$, and $\pi_Y$ be as defined above. For any two curves $\alpha, \beta \in \mathcal{C}^\dagger(S)$ that intersect the subsurface $Y$ essentially with $|\alpha \cap \beta| = p \geq 0$, and for any $\gamma_\alpha \in \pi_Y(\alpha)$ and $\gamma_\beta \in \pi_Y(\beta)$, the following inequality holds:
    $$d_Y^\dagger\big(\gamma_\alpha, \gamma_\beta\big) \leq 8p + 11.$$
\end{proposition}

\begin{proof}
By choosing the regular neighbourhoods in Definition \ref{def_subsur_proj} sufficiently small, we can find curves $\widehat{\gamma_\alpha} \in \pi_Y(\alpha)$ and $\widehat{\gamma_\beta} \in \pi_Y(\beta)$ such that they are transverse within $Y$ and satisfy $|\widehat{\gamma_\alpha} \cap \widehat{\gamma_\beta}| \leq 4 + 4p$. Additionally, we ensure that $\gamma_\alpha \cap \widehat{\gamma_\alpha} = \emptyset$ and $\gamma_\beta \cap \widehat{\gamma_\beta} = \emptyset$, which implies $d_Y^\dagger(\gamma_\alpha, \widehat{\gamma_\alpha}) = d_Y^\dagger(\gamma_\beta, \widehat{\gamma_\beta}) = 1$. By applying Proposition \ref{lem_upper_dist}, we obtain $d_Y^\dagger\big(\widehat{\gamma_\alpha}, \widehat{\gamma_\beta}\big) \leq 8p + 9$. Then, using the triangle inequality, we conclude that $d_Y^\dagger\big(\gamma_\alpha, \gamma_\beta\big) \leq 8p + 11$.
\end{proof}

Using the triangle inequality and applying Proposition \ref{prop_disjoint_subsurface_proj}, we can further deduce the following corollary:
\begin{corollary}
    Let $P = \left\{\alpha_0, \alpha_1, \ldots, \alpha_n\right\}$ be a sequence in $\mathcal{C}^{\dagger}(S)$ and $p>0$ be a positive integer. Suppose that every $\alpha_i$ intersects $Y$ essentially and that $|\alpha_i\cap \alpha_{i+1}|\leq p$, then $d^{\dagger}_Y\left(\pi_Y(\alpha_0), \pi_Y(\alpha_n)\right) \leq (8p+11) n$.
\end{corollary}

\subsection{Surgery.}\label{subsec_surgery} To prove Theorem \ref{thm_bounded_geodesic_image}, we will employ the auxiliary process of subsurface surgery, similar to the classical cases. The idea is to adapt the proof from \cite{webb2013shortproofboundedgeodesic} to the case of the fine curve graph.

First, let us introduce the following notion:

\begin{definition}[Sensible collection]
    Let $S$ be a hyperbolic surface, possibly with punctures. We say that a collection of essential multicurves $\left(\gamma_i\right)\subset \mathcal{C}^\dagger(S)$ is {\it sensible} if the curves are essential and have no triple points, {\it i.e.} for distinct indices $i, j, k$, we have $\gamma_i \cap \gamma_j \cap \gamma_k = \emptyset$. 
\end{definition}

Throughout this section, we will assume that $\alpha, \beta \in \mathcal{C}^\dagger(S)$ are two curves in minimal position that fill $S$, {\it i.e.} cut $S$ into a disjoint union of discs.

Recall that when we orient $\gamma$ and $\beta$, each point in $\gamma \cap \beta$ is assigned an intersection sign of $\pm 1$. We say that a pair of such points have opposite signs if their signs differ, and the same sign otherwise. Note that this notion does not depend on the choice of orientation.

\begin{definition}[Loop]\label{defn_loop}
    Let $S$ be a surface and $\alpha,\beta$ be as above. We say that $\gamma$ is an {\it $(\alpha, \beta)$-loop} if for each subarc $ b \subset \beta-\alpha$, we have $|\gamma \cap b| \leq 2$, with equality only if $\gamma \cap b$ have opposite sign.
\end{definition}

In the following paragraphs, we will assume only that $\alpha$ is a multicurve, while $\beta$ and $\gamma$ are essential curves on $S$. Given a sensible triplet $\{\gamma, \alpha, \beta\}$, we aim to construct a new curve $\chi(\gamma)$ that forms an $(\alpha, \beta)$-loop if $\gamma$ does not already do so.

Suppose now that $\{\gamma, \alpha, \beta\}$ is a sensible triplet and that $\gamma$ is not an $(\alpha, \beta)$-loop. Let $c$ be a minimal connected subarc of $\gamma$, with respect to inclusion, such that there exists an arc $b \subset \beta - \alpha$ satisfying one of the following conditions:
\begin{itemize}[topsep=0pt, itemsep=-1ex, partopsep=1ex, parsep=1ex]
    \item $c \cap b$ consists of a pair of points with the same sign, or
    \item $|c \cap b| \geq 3$.
\end{itemize}
Since $c$ is minimal and connected, this implies that $c$ has endpoints on $b$, and that $b$ is the unique arc satisfying the above properties, with the additional condition that $|c \cap b| \leq 3$.

Now the intersection between $b$ and $c$ falls into the following three cases, with the same surgery as in \cite{webb2013shortproofboundedgeodesic} recorded here below. Let $\mathcal{N}$ be a small enough regular neighbourhood of $\beta$ and $\mathcal{R}\subset \mathcal{N}-\alpha$ a rectangle containing $b$. 
\vskip 0.3cm
\begin{enumerate}[label=(K\arabic*),topsep=0pt, itemsep=-1ex, partopsep=1ex, parsep=1ex]
    \item $|c\cap b|=2$ with same sign. Then let $\{p_1,p_2\}=c\cap \partial\mathcal{R}$. Connect $p_1$ to $p_2$ (or {\it vice versa}) by an arc $a \subset \mathcal{R}$ that intersects $b$ only once and intersects $c$ only at the endpoints of $a$. We let $\chi(\gamma)$ be the simple closed curve $a \cup(c-\mathcal{R})$. See Figure \ref{fig 2}\label{K1}
    \item $|c \cap b|=3$ with alternating intersection signs with respect to some order on $b$. Let $p_1, p_2, p_3$ be the three points $c \cap b$ in some order along $b$. Let $c_1, c_2 \subset c$ be arcs such that $c_1 \cup c_2=c$, $\partial c_1=\left\{p_1, p_2\right\}$ and $\partial c_2=\left\{p_2, p_3\right\}$. Connect $c_1 \cap \partial R$ to $c_2 \cap \partial \mathcal{R}$ by two disjoint arcs $a_1, a_2 \subset \mathcal{R}$ so that $a_1$ intersects $c_1, c_2$ only at its endpoints and intersects $b$ once, and similarly $a_2$. We let $\chi(\gamma)=a_1 \cup\left(c_1-\mathcal{R}\right) \cup a_2 \cup\left(c_2-\mathcal{R}\right)$.\label{K2}
    \item $|c \cap b|=3$ with non-alternating intersection signs. Let $p_1, p_2, p_3$ be the three points $c \cap b$ in some order along $b$. Without loss of generality, assume that $p_1$ and $p_2$ have the same sign. Then connect $p_1$ to $p_2$ with an arc $a$ similarly to \ref{K1} so that the sign at the unique intersection point between $a$ and $b$ is different from the intersection sign of $p_3$.\label{K3}
\end{enumerate}
We denote by $\chi(\gamma)$ this curve newly constructed from $\gamma$ according to the aforementioned procedures.

\begin{figure}[!ht]
    \centering
    \begin{multicols}{2}
    \includegraphics[width=1\linewidth]{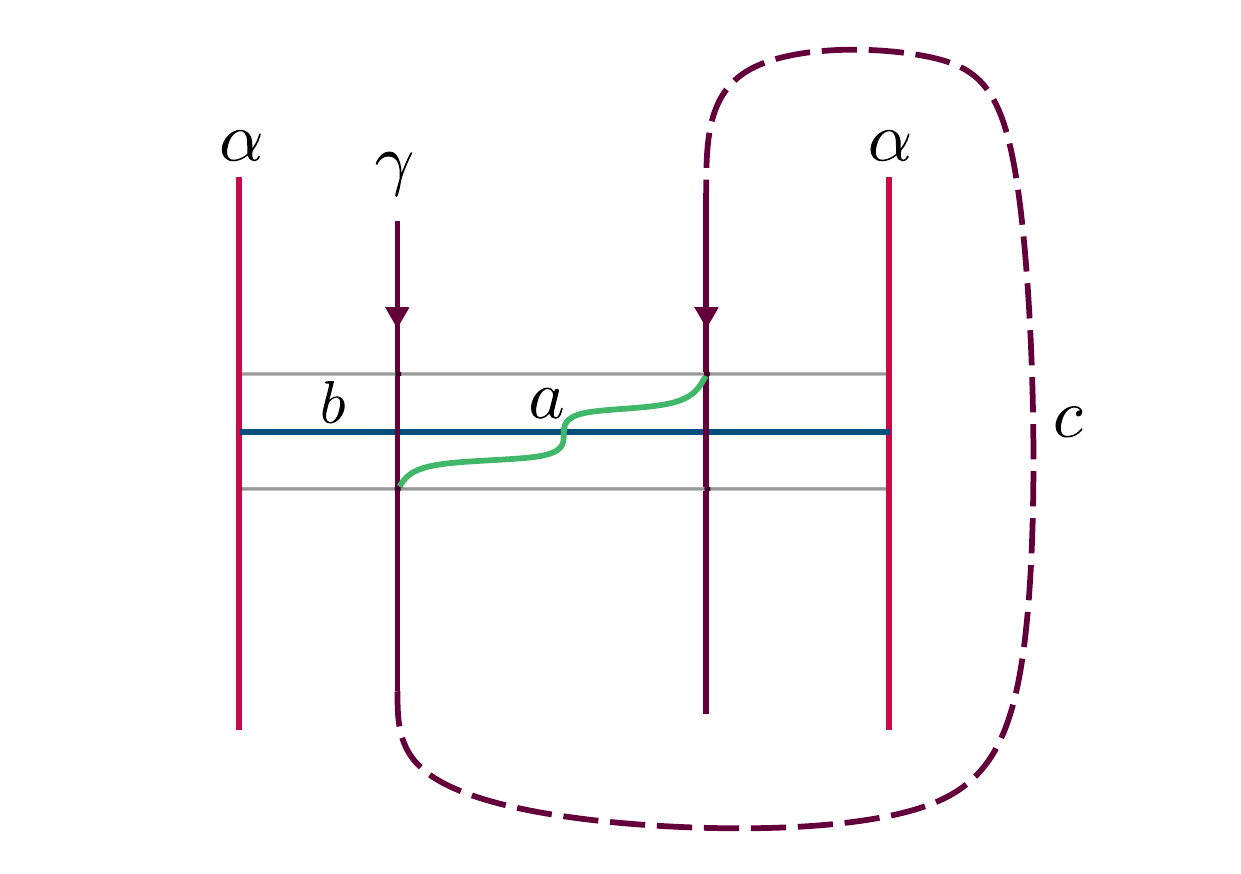}\par
    \includegraphics[width=1\linewidth]{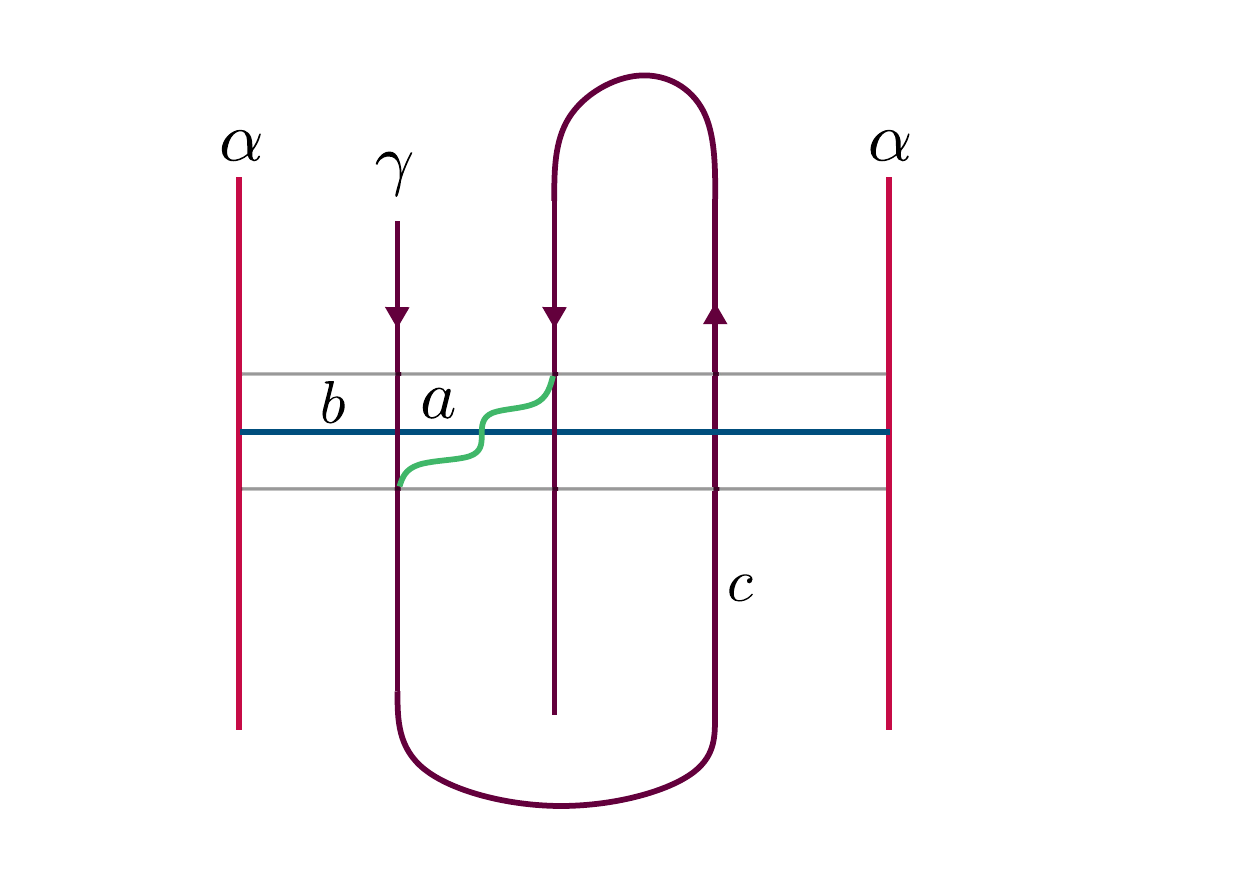}
    \end{multicols}
    \caption{Surgery for \ref{K1} (left) and \ref{K3} (right).}
    \label{fig 2}
\end{figure}

Now let $S$ be again an oriented closed surface with genus $g(S)\geq 2$. Let $\beta\in\mathcal{C}^\dagger(S)$ be a curve and $\alpha$ be a multicurve that is transverse to $\beta$. Suppose that $P\subset S$ is a finite collection of points on the surface such that $\alpha,\beta$ are in minimal position on $S-P$. Moreover, we assume that $p \in P$ if and only if $p$ is contained in a bigon between $\alpha$ and $\beta$ and that each bigon contains only one point in $P$.

\begin{lemma}\label{lem_surgery_surviving}
Let $S$, $\alpha$, $\beta$ and $P$ be as above. Suppose that $\gamma\in \mathcal{C}^\dagger(S)$ is a curve such that the triplet $(\gamma,\alpha,\beta)$ is sensible on $S-P$. Then we can find $\chi(\gamma)\in \mathcal{C}^\dagger(S-P)$ is an $(\alpha,\beta)$-loop and is surviving, {\it i.e.} $\chi(\gamma)\in \mathcal{C}^\dagger(S)$. 
\end{lemma}
\begin{proof}
The same proof in \cite[Lemma 2.2]{webb2013shortproofboundedgeodesic} also applies here, resulting that $\chi(\gamma)$ is an $(\alpha,\beta)$-loop and remains essential on $S-P$. To show that $\chi(\gamma)$ is surviving, it suffices to show that $\chi(\gamma)$ can be taken so that they do not only bound a disk in $S$ containing some points in $P$. 

For the cases \ref{K1} and \ref{K3}, by proof of contradiction, we assume that $\chi(\gamma)$ is not surviving. Then $\chi(\gamma)$ encloses a punctured disk $\Delta$. There are two cases: $\gamma_1 \subset \Delta$ or $\gamma_2 \subset \Delta$. In either case, the extension of $\gamma_i$ would intersect with $\partial \Delta$. This contradicts the fact that $\gamma$ is a simple closed curve.

For the case \ref{K2}, assume also {\it ab absurdo} that $\chi(\gamma)$ is not surviving. Observe that $\chi(\gamma)$ has to be has to be separating, since a non-separating curve cannot cut an oriented surface into a disk. This means that the two subsurfaces \(B_\pm\) bounded \(c_1, c_2\) and \(b\) are two disks on \(S\), see Figure \ref{fig 4}. Moreover, by non-surviving condition, it is implied that $B_-$ and $B_+$ both contains points in $P$. By our assumption on $P$, this further indicates that both $B_-$ and $B_+$ are contained in a bigon between $\beta$ and $\alpha$. As a result, it forces a component of $\alpha$ to bound a two bigons and become non-essential on $S$, which yields a contradiction.
\end{proof}

\begin{figure}[!ht]
    \centering
    \includegraphics[width=0.6\linewidth]{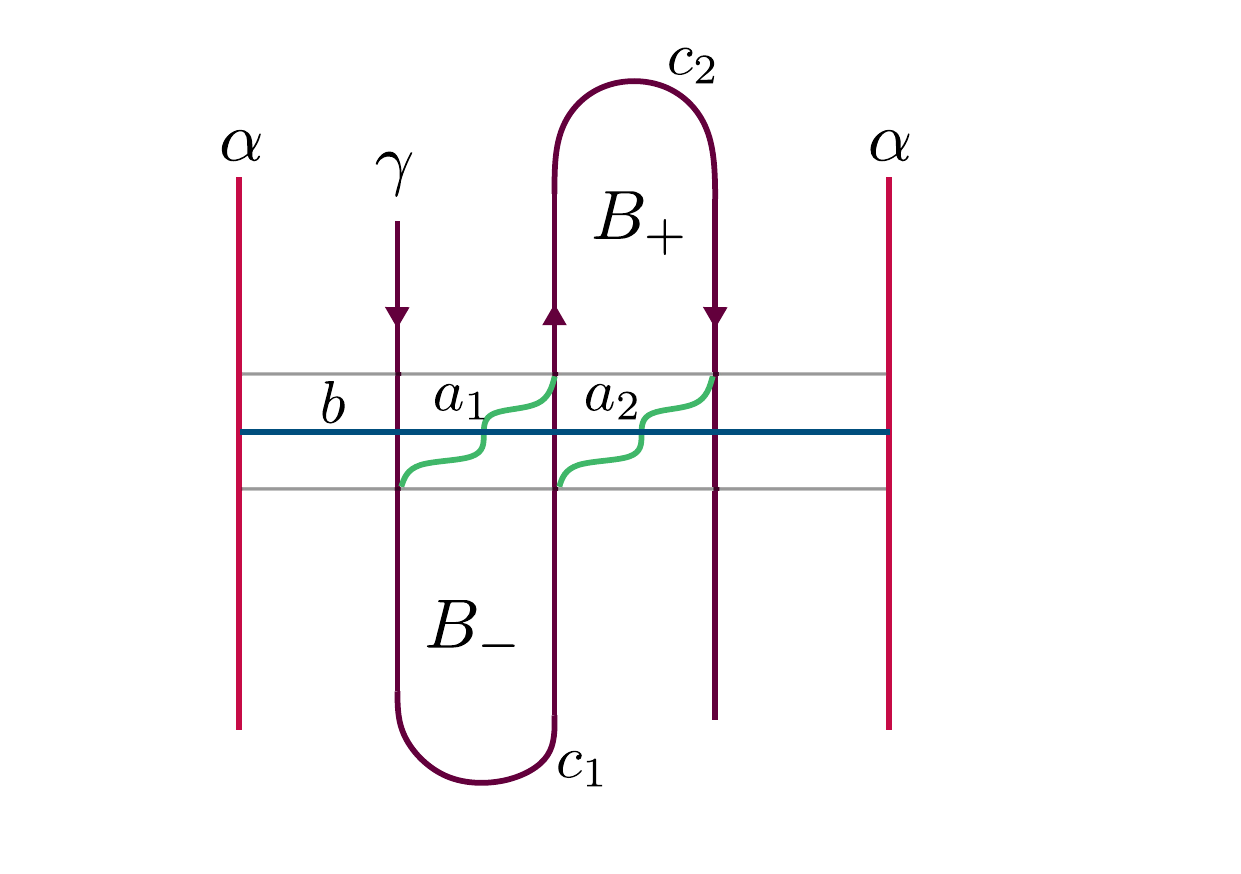}
    \caption{Surgery for \ref{K2} and its proof of Lemma \ref{lem_surgery_surviving}.}
    \label{fig 4}
\end{figure}

\subsection{Proof of bounded geodesic image theorem.} Since the surgery is now well-defined in the context of the fine curve graph, to prove the bounded geodesic image theorem for the fine curve graph, it suffices to proceed through the lemmata below, following the arguments in \cite{webb2013shortproofboundedgeodesic}.

We note that the surgery for \(\gamma\) is performed exactly as in \cite{webb2013shortproofboundedgeodesic}. Hence we may also conclude the following generalised result from \cite[Proposition 3.1.7]{hensel2013slimunicornsuniformhyperbolicity}:
\begin{lemma}[cf. Lemma 2.3, \cite{webb2013shortproofboundedgeodesic}]\label{lem_2.3}
    Let $S$, $P$, $\alpha$ and $\beta$ be as above. Let $\gamma_1,\gamma_2\in\mathcal{C}^\dagger(S)$ be two curves such that $\gamma_1, \gamma_2, \alpha,\beta$ are sensible on $S-P$. Assume in addition that $\gamma_1$ is disjoint from $\gamma_2$. Then for any $\chi(\gamma)$ and $\chi(\beta)$ constructed above, we have then $|\chi(\gamma_1)\cap\chi(\gamma_2)|\leq 4$.
\end{lemma}

The above result will allow us to deduce the following distance estimation between the subsurface projection of $(\alpha, \beta)$-loops obtained from two disjoint curves from Proposition \ref{lem_upper_dist}:
\begin{lemma}
Let $S$, $P$, $\alpha$, $\beta$, $\gamma_1$ and $\gamma_2$ be as above. Then $d^\dagger\big(\chi(\gamma_1), \chi(\gamma_2)\big)\leq 9$.
\end{lemma}

With results provided above, following the arguments of \cite[Lemma 2.4]{webb2013shortproofboundedgeodesic}, we can obtain the following result:
\begin{lemma}\label{lem_2.4}
    Let $\alpha^{\prime}$ be a component of a multicurve $\alpha$ on $S$ and let $\beta$ be a curve on $S$. Suppose $\alpha^{\prime}, \beta$ are sensible and fill $S$. Then there exists a $(9,0)$-quasigeodesic $\alpha^{\prime}=\gamma_0, \gamma_1, \ldots, \gamma_n=$ $\beta$ with $\gamma_i$ is a $(\alpha, \beta)$-loop for every $0<i<n$.
\end{lemma}

Moreover, we have the following result:
\begin{lemma}\label{lem_2.5}
      Let $Y$ be a non-annular subsurface of $S$ and $\beta\in \mathcal{C}^\dagger(S)$. Suppose in addition that a component $a\subset\partial Y$ and $\beta$ fill $S$. Let $\gamma\in\mathcal{C}^\dagger(S)$ be a $(\partial Y, \beta)$-loop that intersects $\partial Y$. Then for any $\widehat{\gamma}\in \pi_Y(\gamma)$ and $\widehat{\beta}\in \pi_Y(\beta)$, we have $d_Y(\widehat{\gamma}, \widehat{\beta}) \leq 24$.
\end{lemma}
\begin{proof}
    Since $Y$ is non-annular subsurface, any pair of arcs in the respective subsurface projection will intersect at most twice by Definition \ref{defn_loop}, so one can consider a closed regular neighbourhood of the arcs to prove that there are some $\gamma'\in \pi_{Y}(\gamma)$ and $\beta'\in \pi_{Y}(\beta)$ satisfying $d_Y(\gamma',\beta')\leq 2$. Applying triangle inequality and Proposition \ref{prop_ssurface_proj_well_defined}, we can conclude the result.
\end{proof}

With the lemmata above, we can now prove Theorem \ref{thm_bounded_geodesic_image}.
\begin{proof}[Proof of Theorem \ref{thm_bounded_geodesic_image}]
Since we have finished all the preparation using surgery for the proof of bounded geodesic image theorem. Note that the fine curve graph is also Gromov hyperbolic, the argument from \cite{webb2013shortproofboundedgeodesic} also applies here to conclude Theorem \ref{thm_bounded_geodesic_image}.
\end{proof}

\subsection{A surviving version.} Assume that $S$ is a surface as given above and $P\subset S$ is a finite collection of point on $S$. Let $Y\subset S-P$ be a proper subsurface that is not sporadic. We consider the subsurface projection $\widetilde{\pi_Y}\colon \mathcal{C}^s(S-P)\to \mathcal{P}(\mathcal{C}(Y))$ given as in the classical cases, {\it i.e.} the image of an (isotopy class of a) surviving curve is the isotopy class of curves defined as in \S\ref{subsec_s_p}. 

We observe that the only arguments in the proof of the bounded geodesic image theorem that depend on the topological type of the curves are the surgeries performed in Proposition \ref{lem_upper_dist} and in \S\ref{subsec_surgery}. As shown in Proposition \ref{lem_upper_dist} and Lemma \ref{lem_surgery_surviving}, the new curves obtained from the surgeries remain surviving on $S-P$. Recall that the surviving curve graph $\mathcal{C}^s(S-P)$ is also Gromov hyperbolic (see \cite{rasmussen2020uniform} and \cite[theorem 3.3]{bowden2022quasi}). Following the approach in \cite{webb2013shortproofboundedgeodesic}, we conclude:
\begin{theorem}\label{thm_surviving_bounded_geodesic_image}
Let $S$ be a closed surface and $P\subset S$ be a finite collection of point. Then there exists an $M\geq 0$ such that for any non-sporadic proper subsurface $Y\subset S-P$ and any geodesic $(g_i)\subset \mathcal{C}^s(S-P)$ where each $g_i$ intersects $Y$ essentially, we have $\diam_{d_Y}\big(\widehat{\pi_Y}(g_i)\big)\leq M$.
\end{theorem}

\begin{remark}
As both $\pi_Y$ and $\widehat{\pi_Y}$ are coarsely well-defined for a curve $\alpha$ intersecting $Y$ essentially, we will simply write $\pi_Y(\alpha)$ and $\widehat{\pi_Y}(\alpha)$ for an (arbitrary) element in the projection.
\end{remark}

\subsection{Application: non-compactness.} As an interesting application of bounded geodesic image theorem, we can show that the Gromov boundary and the Gromov bordification of the fine curve graph are not compact.
\begin{proof}[Proof of Theorem \ref{thm_cpt}]
Let $o\in \NC^\dagger(S)$ and let $\xi\in \partial\mathcal{C}^\dagger(S)$ be a point on the Gromov boundary. Let $Y=S-o$ be the complementary component and $a\in \mathcal{C}^\dagger(S)$ that intersects $Y$ essentially and $\langle \xi, a\rangle_o\gg 1$. Take $\varphi\in \mathrm{Homeo}(S)$ be a homeomorphism on $S$ such that $\varphi|_Y$ is pseudo-Anosov on $Y$.

We claim that the sequence $\big(\varphi^n(\xi)\big)_{n\geq 1}$ does not have an accumulation point on $\partial \mathcal{C}^\dagger(S)$. Since $\partial \mathcal{C}^\dagger(S)$ is metrisable, suppose for contradiction that there is a convergent subsequence $\big(\varphi^{n_k}(\xi)\big)_{k\geq 1}$. By (\ref{product_est}), this implies that for any $R>0$, there exists a sufficiently large $k_0>0$ so that for an arbitrary $k\gg k_0$, we have $\langle \varphi^{n_{k_0}}(a),\varphi^{n_k}(a)\rangle_o> R$. We may assume that $\varphi^{n_{k_0}}(a),\varphi^{n_k}(a)$ are transverse to $o$ (as such a change causes an additive error at most $2$ in the Gromov product). Let $P\subset S$ such that $\varphi^{n_{k_0}}(a)$, $\varphi^{n_k}(a)$ and $o$ are pairwise in minimal position. Since $k\gg k_0$ and $\varphi$ is pseudo-Anosov on $Y$, we have then 
$$d_{Y}\big(\pi_Y\big(\varphi^{n_{k_0}}(a)\big),\big(\varphi^{n_{k}}(a)\big)\big)>M,$$
where $M>0$ is the constant from Theorem \ref{thm_surviving_bounded_geodesic_image}. Appealing to Theorem \ref{thm_surviving_bounded_geodesic_image}, this means that any geodesic connecting $[\varphi^{n_{k_0}}(a)]_{S-P}$ to $[\varphi^{n_{k}}(a)]_{S-P}$ must contain an isotopy class of curve that does not intersect $Y$ essentially, which has to be $[o]_{S-P}$. So by Lemma \ref{prop_NC_approx}, we can compute
\begin{align*}
\langle \varphi^{n_{k_0}}(a),\varphi^{n_{k}}(a)\rangle_o &= \frac{1}{2}\left(d^\dagger\big(\varphi^{n_{k_0}}(a),o\big)+d^\dagger\big(\varphi^{n_{k}}(a),o\big)-d^\dagger\big(\varphi^{n_{k_0}}(a),\varphi^{n_{k}}(a)\big)\right)\\
&=\frac{1}{2}\Big(d_{C^s(S-P)}\big([\varphi^{n_{k_0}}(a)]_{S-P},[o]_{S-P}\big)+d^\dagger\big([\varphi^{n_{k}}(a)]_{S-P},[o]_{S-P}\big)\\
&\qquad -d^\dagger\big([\varphi^{n_{k_0}}(a)]_{S-P},[\varphi^{n_{k}}(a)]_{S-P}\big)\Big)\\
&=0<R,
\end{align*}
contradiction! Hence $\partial\mathcal{C}^\dagger(S)$ is not sequentially compact. As it is a metrisable space, it is not compact.
\end{proof}

Note that the bordification of the fine curve graph is also completely metrisable (cf. \cite[Corollary 3.6.14]{das2017geometry}). So if it is compact, it will also be sequentially compact. But the sequence $\big(\varphi^n(a)\big)_{n\geq 1}$ as in the proof of Theorem \ref{thm_cpt} does not contain a subsequence converging in in $\mathrm{bord}\, \mathcal{C}^\dagger(S)$, otherwise the Gromov product between two elements from the subsequence has to converge in $(0,\infty]$, which is not the case. The existence of such a sequence implies again that the bordification is not compact:
\begin{corollary}
Let $S$ be an orientable connected closed surface with genus $g\geq 2$. Then the Gromov bordification $\mathrm{bord}\, \mathcal{C}^\dagger(S)$ is not compact.
\end{corollary}

%% file: sections/sec-5.tex
\section{Connectivity properties}\label{sec-5}

\subsection{Closed surfaces of higher genus.} Throughout this subsection, we will assume that $S$ is a closed surface of genus $g(S)\geq 3$. In this subsection, we will focus on verifying the conditions in Proposition \ref{prop_wright} for the subgraph $\NCo(S)$.

For convenience, let $r>0$ be any positive integer, we denote by $\So(r)$ the sphere in $\NCo(S)$ centred at $o$ of radius $r$, {\it i.e.} 
$$
\So(r) = \{\alpha\in \NCo(S)\colon d_o^\dagger(\alpha,o)=r\}.
$$

The following lemma shows that there is no ``{\it dead ends}'' in the sense of \cite{birman2015curve} inside of the subgraph $\NCo(S)$, {\it i.e.} this subgraph satisfies the assumption \ref{G1} in Proposition \ref{prop_wright}.
\begin{lemma}\label{lem_G1}
Let $S$ be a closed hyperbolic surface. Let $\alpha\in \So(r)$ for some $r>0$. Then there exists a curve $\beta\in \So(r+1)$ such that $\beta\cap \alpha=\emptyset$.
\end{lemma}
\begin{proof}
Let $o$ and $\alpha$ be as above and let $P\subset S$ be a finite collection of point such that $o$ and $\alpha$ are in minimal position of $S-P$. By Proposition \ref{prop_metric}, we can pick a hyperbolic metric on $S-P$ such that $o$ and $\alpha$ are geodesic.

Consider the subsurface $Y$ of $S-P$ given by cutting along $\alpha$. Note that the subsurface projection $\widehat{\pi_Y}\colon \mathcal{C}^s(S-P)\to \mathcal{C}(Y)$ is unbounded, as the subset of non-separating curves on $Y$, {\it a fortiori} surviving on $S-P$, is unbounded. We can take a non-separating geodesic curve $\beta\subset Y$ with respect to the hyperbolic metric on $S-P$ such that $d_Y\big(\widehat{\pi_Y}(o),\widehat{\pi_Y}(\beta)\big)>M$, where $M$ is the constant from Theorem \ref{thm_surviving_bounded_geodesic_image}. Since $\beta$ is non-separating on $S$, it is surviving on $S-P$. Now Theorem \ref{thm_surviving_bounded_geodesic_image} implies that any geodesic in $\mathcal{C}^s(S-P)$ connecting $[o]_{S-P}$ to $[\beta]_{S-P}$ must contain an element that does not intersect $Y$ essentially. But as $\alpha$ is non-separating on $S$ (and {\it a fortiori} on $S-P$), the only curves that do not intersect $Y$ are those isotopic to $\alpha$. This implies that the geodesic in $\mathcal{C}^s(S-P)$ connecting $[o]_{S-P}$ to $[\beta]_{S-P}$ must pass through $[\alpha]_{S-P}$. As $\beta$ is disjoint from $\alpha$ by definition, we have
\begin{align}\label{eq_o,a,b}
    d_{\mathcal{C}^s(S-P)}([o]_{S-P},[\beta]_{S-P})=d_{\mathcal{C}^s(S-P)}([o]_{S-P},[\alpha]_{S-P})+1.
\end{align}
Because $o,\alpha,\beta$ are geodesic curves on $S-P$, they are pairwise in minimal position on $S-P$. Combining Proposition \ref{prop_NC_approx} and (\ref{eq_o,a,b}), we can conclude that $d^\dagger(o,\beta)=d^\dagger(o,\alpha)+1$ on $\mathcal{C}^\dagger(S)$. Since the inclusion map $\NCo(S)\hookrightarrow\mathcal{C}^\dagger(S)$ is an isometric embedding (appealing to Lemma \ref{radial_dist_pres}), we can conclude that $\beta\in \So(r+1)$.
\end{proof}

Now, we will treat in the rest of this subsection the condition \ref{G3} in Proposition \ref{prop_wright}. First, we will need the following auxiliary lemmata:

\begin{lemma}\label{lem_trans_all}
    Let $r > 0$ be any positive integer, and let $\alpha, \beta \in \So(r)$. If $\alpha$ and $\beta$ are not transverse, then there exists a $\beta^{\prime} \in \So(r)$ that is isotopic to $\beta$ yet disjoint from it, and which is also transverse to $\alpha$.
\end{lemma}
\begin{proof}
In light of Proposition \ref{prop_NC_approx}, we can find a finite subset $P\subset S$ so that $(o,\alpha)$ and $(o,\beta)$ are in minimal positions on $S-P$. By making a small perturbation as described in Lemma \ref{lem_pert}, we can obtain a curve $\beta'$ that is transverse to $\alpha$ and $o$, disjoint from $\beta$, and isotopic to $\beta$ in $S- P$ while remains in minimal position with $o$ in $S-P$. With these assumptions, combining Proposition \ref{prop_NC_approx} and Lemma \ref{radial_dist_pres}, we can conclude that 
  $$
  d_{o}^{\dagger}(o,\beta') = d^{\dagger}(o,\beta') = d_{S-P}([o],[\beta']) = d_{S-P}([o],[\beta]) = d^{\dagger}(o,\beta) = d_{o}^{\dagger}(o,\beta) = r.
  $$  
This implies that $\beta'\in \So(r)$ and the proof is complete.
\end{proof}

\begin{remark}
In the sequel, as for any $\alpha\in \NCo(S)$, we have the distance equality between $d^\dagger(o,\alpha)$ and $d^\dagger_o(o,\alpha)$, we will no longer distinguish these two distances when working on the spheres.
\end{remark}

\begin{lemma}\label{lem_iso}
    Let $r>0$ and $\alpha, \widehat{\alpha}\in \So(r)$ be isotopic on $S - P$, where $P$ is a finite subset of $S$. If $(o, \alpha)$ and $(o,\widehat{\alpha})$ are in minimal position on $S - P$, there exists a path on $\So(r)$ connecting $\alpha$ to $\widehat{\alpha}$, in which every vertex is a curve isotopic to $\alpha$ on $S - P$. If in addition, if $\gamma \in \NCo(S)$ is disjoint from both $\alpha$ and $\widehat{\alpha}$, then the path can be chosen so that all vertices on the path are disjoint from $\gamma$.
\end{lemma}
\begin{proof}
    Up to a small perturbation as in Lemma \ref{lem_pert}, we can take a curve $\alpha_1$ that is isotopic $\alpha$ to but disjoint from $\alpha$ and transverse to $\widehat{\alpha}$ (and also to $o$). As the perturbation can be made sufficiently small, one can ensure that there is no bigon between the curve $\alpha_1$ and the curve $o$ on $S-P$, or equivalently they remain in minimal position. If $\alpha_1$ is also disjoint from $\widehat{\alpha}$, the proof is done. Otherwise, if $\alpha_1 \cap \widehat{\alpha} \neq \emptyset$, there is at least one bigon between $\alpha_1$ and $\widehat{\alpha}$ as they are isotopic (so the minimal position is disjoint). Repeating the procedures of isotoping an subarc of $\alpha_1$ to the other side of $\alpha$, we can then construct a path $\alpha_1, \alpha_2,\dots,\alpha_n$ with $\alpha_i\cap \alpha_{i+1}=\emptyset$ for any $1\leq i<n$ and $\alpha_n\cap \widehat{\alpha}=\emptyset$. Moreover, as they remain in minimal position with $o$, by Proposition \ref{prop_NC_approx} and Lemma \ref{radial_dist_pres}, we still have $d_o^\dagger(o,\alpha_i)=r$ for all $i$, {\it i.e.} $\alpha_1, \alpha_2,\dots,\alpha_n$ form a desired path on $\So(r)$ connecting $\alpha$ to $\widehat{\alpha}$. For disjointness from $\gamma$, as $\gamma$ and $\alpha,\widehat{\alpha}$ contain {\it a priori} no bigon, at each step of the isotopy above, the newly constructed curve will not produce bigon with $\gamma$, {\it i.e.} they remain disjoint from $\gamma$.
\end{proof}

Lemma \ref{lem_iso} indicates that under the assumption of condition \ref{G2}, {\it i.e.} two curves $\alpha,\beta\in \So(r+1)$ both adjacent to a curve $\gamma\in\So(r)$, then we can connect $\alpha$ to $\alpha'$ (respectively $\beta$ to $\beta'$) by a path on $\So(r+1)$ where each vertex is again disjoint from $\gamma$, so that $o,\gamma,\alpha',\beta'$ are geodesic curves on a punctured surface $S$. In result, these curves all lie in an isometric copy of $\mathcal{C}^s(S-P)\hookrightarrow \mathcal{C}^\dagger(S)$, allowing us to deduce the condition \ref{G2} with $D=2$ as \cite{wright2023spheres}.

To be precise, we will prove the following result:
\begin{lemma}\label{lem_G2}
Let $S$ be a closed surface of genus $g\geq 3$ and let $\gamma\in \So(r)$ for some positive integer $r>0$ and $\alpha,\beta\in \So(r+1)$ that are adjacent to $\gamma$ in $\NCo(S)$. Then there exists a path 
$$\alpha=\nu_0,\nu_1,\dots,\nu_\ell=\beta$$
in $\So(r+1)$ connecting $\alpha$ to $\beta$ with $\nu_i\cap \gamma =\emptyset$.
\end{lemma}
\begin{proof}
Let $P\subset S$ be a finite collection of points such that the pairs $(o,\gamma)$, $(o,\alpha)$ and $(o,\beta)$ are in minimal position on $S-P$. Suppose that $S-P$ is carrying a hyperbolic metric such that $o$ and $\gamma$ are geodesic, of which the existence is guaranteed by Proposition \ref{prop_metric}. Let $\alpha'$ and $\beta'$ be the geodesic representative of $[\alpha]_{S-P}$ and $[\beta]_{S-P}$. By Lemma \ref{lem_iso}, we can already find paths on $\So(r+1)$ connecting $\alpha$ to $\alpha'$ and $\beta$ to $\beta'$ such that every vertices on the two paths are disjoint from $\gamma$. So to complete the proof, it suffices to prove the lemma for $\alpha'$ and $\beta'$.

Let $Y$ be the complementary component of $\gamma$ in $S-P$. Let $K>0$ be a positive integer, take
$$\mathcal{O}(\gamma,K)\coloneq \left\{\nu\in \So(r+1)\colon d_Y\big(\widehat{\pi_Y}(o),\widehat{\pi_Y}(\nu)\big)>K\right\}.$$
By appealing to the proof of \cite[Lemma 7.2]{wright2023spheres}, we can show that for any $K> M$, where $M$ is the constant from Theorem \ref{thm_surviving_bounded_geodesic_image}, we can find a path connecting $\alpha'$ to $\alpha''\in \mathcal{O}(\gamma,K)$ such that each vertices on the path is on $\So(r+1)$ and disjoint from $\gamma$. We remark that whenever the classical bounded geodesic image theorem is used in the proof of \cite[Lemma 7.2]{wright2023spheres}, we will replace it by Theorem \ref{thm_surviving_bounded_geodesic_image}. Note that $Y$ has genus at least $2$ and $[\alpha]_Y$ is non-separating on $Y$. Inspecting the proof of \cite[Lemma 7.4]{wright2023spheres}, the path above can be chosen such that all vertices are non-separating curves on $Y$ and {\it a fortiori} on $S$. The same can be done for $\beta'$ connecting to a $\beta''\in \mathcal{O}(\gamma,K)$. Moreover, as $\alpha''$ and $\beta''$ are disjoint from $\gamma$, their $Y$-subsurface projection consists only their isotopy class on $Y$, {\it i.e.} $[\alpha'']_Y,[\beta'']_Y\in \mathcal{C}(Y)$.

Now it remains to show that $\mathcal{O}(\gamma,K)$ is connected for some sufficiently large $K>M$. Use the connectedness of spheres in the curve graph of high complexity surfaces \cite[Theorem 1.1]{wright2023spheres}, we can obtain a path in $\mathcal{C}(Y)$
$$[\alpha'']_Y=q_0,\dots, q_n=[\beta'']_Y$$
such that $d_Y\big(q_i,\widehat{\pi_Y}(o)\big)>K>M+C$, where $C$ is the constant in \cite[Lemma 3.11]{wright2023spheres} for the existence of a path connecting two non-separating curves on $Y$, within distance $C$ from the geodesic in $\mathcal{C}(Y)$ between theses two curves, consisting of only isotopy classes of non-separating curves on $Y$, {\it i.e.} \cite[Lemma 3.11]{wright2023spheres} implies that we can modified the above path into a new path
$$[\alpha'']_Y=q_0',\dots, q_m'=[\beta'']_Y$$
where $q_i'$ are non-separating curves on $Y$ such that $d_Y\big(q_i',\widehat{\pi_Y}(o)\big)>M$. By descending to their geodesic representatives on $S-P$, Theorem \ref{thm_surviving_bounded_geodesic_image} applies to make it a path on $\So(r+1)$ connecting $\alpha''$ to $\beta''$ such that all vertices are disjoint from $\gamma$.

Now the concatenating path $\alpha\rightsquigarrow \alpha'\rightsquigarrow \alpha''\rightsquigarrow \beta''\rightsquigarrow \beta'\rightsquigarrow \beta$ will yield the desired path.
\end{proof}

\begin{lemma}\label{lem_G3}
Let $S$ be a closed surface of genus $g(S)\geq 3$. For every adjacent pair $\alpha, \beta \in \So(r)$, there exists a path
$$\alpha=\nu_0,\ \nu_1,\ \cdots,\ \nu_\ell=\beta$$
such that $\nu_i\in \So(r+1)$ for $0<i<\ell$ and $d^\dagger(\nu_i,\alpha)\leq 2$.
\end{lemma}
\begin{proof}
Let $P\subset S$ be a finite collection of points such that $o,\alpha,\beta$ are pairwise in minimal position.

\begin{enumerate}[label=(Case\ \arabic*),topsep=0pt, itemsep=-1ex, partopsep=1ex, parsep=1ex]
\item \label{case 1} First suppose that $(\alpha\cup\beta)$ is either isotopic or non-separating on $S$. Then let $Y$ be the complementary component in $S-P$ of $\alpha\cup \beta$ that is not a punctured cylinder. Again, we can find a non-separating curve $\gamma$ on $Y$ such that $d_Y\big(\widehat{\pi_Y}(o),\widehat{\pi_Y}(\gamma)\big)>M$ for $M>0$ from Theorem \ref{thm_surviving_bounded_geodesic_image}. So any geodesic in $\mathcal{C}^s(S-P)$ connecting $[o]_{S-P}$ to $[\gamma]_{S-P}$ must contain an element that does not intersect $Y$ essentially, whereas the only isotopy classes of surviving curves that do not intersecting $Y$ essentially are $[\alpha]_{S-P}$ to $[\beta]_{S-P}$. Note that $\gamma$ is surviving on $S-P$ and disjoint from $\alpha,\beta$. By making $\gamma$ in minimal position with $o$ on $S-P$, similar to the case in Lemma \ref{lem_G1}, we can conclude that $\gamma \in \So(r+1)$.

\item Now suppose that $(\alpha\cup\beta)$ is separating and let $Y_1,Y_2\subset S-(\alpha\cup \beta)$ be two complementary components. As $g(S)\geq 3$, we observe that $Y_1$ has genus $g(Y_1)\geq 1$. Similarly, take $\gamma\subset Y_1$ be a geodesic non-separating curve on $Y_1$ such that $d_{Y_1}\big(\pi_{Y_1}(o),[\gamma]_{Y_1}\big)>M$, where $M>0$ is the constant from Theorem \ref{thm_surviving_bounded_geodesic_image}. We claim that $d_{\mathcal{C}^s(S-P)}([o]_{S-P},[\gamma]_{S-P})\neq r-1$. Indeed, by Theorem \ref{thm_surviving_bounded_geodesic_image}, the geodesic in $\mathcal{C}^s(S-P)$ connecting $[o]_{S-P}$ to $[\gamma]_{S-P}$ must contain an element $[\xi]_{S-P}$ that dos not intersecting $Y_1$ essentially. If $d_{\mathcal{C}^s(S-P)}([o]_{S-P},[\gamma]_{S-P})=r-1$, then $d_{\mathcal{C}^s(S-P)}([o]_{S-P},[\gamma]_{S-P})\leq r-2$ and as a result, $[\xi]_{S-P}\neq [\alpha]_{S-P},[\beta]_{S-P}$ as $d_{\mathcal{C}^s(S-P)}([o]_{S-P},[\alpha]_{S-P})=d_{\mathcal{C}^s(S-P)}([o]_{S-P},[\beta]_{S-P})=r$. This forces that the geodesic representative $\xi$ to be contained in $Y_2$ and hence disjoint from $\alpha$ and $\beta$, making the concatenating path
$$[o]_{S-P}\rightsquigarrow [\xi]_{S-P} \text{ --- } [\alpha]_{S-P}$$
of length at most $r-1$. This contradicts our assumption that $\alpha\in\So(r)$. Hence $d^\dagger(o,\gamma)=d_{\mathcal{C}^s(S-P)}([o]_{S-P},[\gamma]_{S-P})=r$ or $r+1$. If $d^\dagger(o,\gamma)=r+1$, then we have done the proof. If $d^\dagger(o,\gamma)=r$, then $(\alpha\cup \gamma)$ and $(\beta\cup \gamma)$ are non-separating and appealing to \ref{case 1}, the two pairs yield $\nu_1,\nu_2\in \So(r+1)$ both adjacent to $\gamma\in \So(r)$. Now applying Lemma \ref{G2} to the triplet $(\nu_1,\nu_2,\gamma)$, we can get a path $\nu_1\rightsquigarrow \nu_2$ on $\So(r+1)$ such that all vertices are disjoint from $\gamma$. As a result, the concatenating path
$$\alpha \text{ --- }\nu_1 \rightsquigarrow \nu_2 \text{ --- }\beta$$
yields a path on $\So(r+1)$ connecting $\alpha$ to $\beta$ within $B_\alpha(2)$.
\end{enumerate}
Combining the two cases together proves the desired result.
\end{proof} 

Now we can deduce Theorem \ref{main_thm} for closed surfaces with higher genus.
\begin{proof}[Proof of Theorem \ref{main_thm}]
Indeed, $\partial\NCo(S)$ is connected and linear connected by Proposition \ref{prop_wright} for closed surfaces with genus $g(S)\geq 3$, as Lemma \ref{lem_G1}, Lemma \ref{lem_G2} and Lemma \ref{lem_G3} apply to verifies respectively the conditions \ref{G1}, \ref{G2} and \ref{G3}. Note that by Proposition \ref{prop_incl}, $\NCo(S)$ is quasi-isometric to $\NC^\dagger(S)$ and as a result $\partial\NCo(S)\simeq\partial \NC^\dagger(S)$. Finally, we remark that for a closed surface with genus $\geq 3$, we know that $\NC^\dagger(S)$ is quasi-isometric to $\mathcal{C}^\dagger(S)$ with $\partial \NC^\dagger(S)\simeq \mathcal{C}^\dagger(S)$ (see \cite{bowden2022quasi}). This concludes the desired result.
\end{proof}

\begin{remark}
If the base point $o$ is chosen as a separating curve on $S$, then instead of linear connectedness, we will obtain a {\it polynomial connectedness}, {\it i.e.} any two points on the boundary $\xi,\eta\in \partial\mathcal{C}^\dagger(S)$ can be included in a connected compact set with diameter at most $L\rho_{o,b}(\xi,\eta)+(L+1)b$ for some $L>0$, since the visual metrics verify the inequality $|\rho_{o,b}-\rho_{o',b}|\leq bd^\dagger(o,o')$.
\end{remark}

\subsection{Connectedness of the spheres.} In this subsection, we address the connectedness properties of spheres within the fine curve graph, building on techniques developed in \cite{wright2023spheres}:

\begin{lemma}[Lemma 2.1, \cite{wright2023spheres}]\label{lem_wright}
Let $\Gamma$ be an arbitrary graph and $o\in \Gamma$ be arbitrary. Fix $w>0$. Let $S_r$ be the $r$-sphere in $\Gamma$. Suppose in addition that for every $r>0$, the following conditions hold:
\begin{enumerate}[label=(G\arabic*')]
    \setcounter{enumi}{1}
    \item For every $z \in S_r$ and $x, y \in S_{r+1} \cap B_z(1)$ there exists a path
        $$x=x_0,\ x_1,\ \dots,\ x_{\ell}=y$$
        with
        $$x_i \in S_{r+1}\cup \dots \cup S_{r+w}$$
        for $0 \leq i \leq \ell$. \label{G2'}
    \item For every adjacent pair $x, y \in S_r$ there exists a path
        $$x=x_0,\ x_1,\ \dots,\ x_{\ell}=y$$
        with
        $$x_i \in S_{r+1}\cup \dots \cup S_{r+w}$$
    for $0<i<\ell$. \label{G3'}
Then for any $r>0$, the union $S_{r}\cup \dots \cup S_{r+w-1}$ is connected.
\end{enumerate}
\end{lemma}
\begin{remark}
We remark that the condition \ref{G2} (respectively \ref{G3}) is stronger than the condition \ref{G2'} (respectively \ref{G3'}).
\end{remark}

In order to show the connectedness of the spheres in $\mathcal{C}^\dagger(S)$, we will first check the connectedness of the spheres in $\NCo(S)$.

\begin{lemma}\label{lem_connect}
Let $S$ be a surface as given above. For any positive integer $r>0$, the sphere $\So(r)$ is connected.
\end{lemma}
\begin{proof}
Lemma \ref{lem_G2} and Lemma \ref{lem_G3} show that the graph $\NCo(S)$ verifies \ref{G2} and \ref{G3}, {\it a fortiori}  \ref{G2'} and \ref{G3'}. By Lemma \ref{lem_wright}, the sphere $\So(r)$ is connected.
\end{proof}

First, we observe that, by Lemma \ref{radial_dist_pres},
\begin{equation}\label{observation}
    \So(r) \subset \mathcal{S}_o(r) = \{\alpha\in \NC^\dagger(S) \colon d^\dagger(\alpha,o)=r\}.
\end{equation}
Hence the sphere $\So(r)$ forms a subgraph in the sphere $\mathcal{S}_o(r)$. Although Proposition \ref{prop_incl} already shows that $\NCo(S)$ is $1$-dense in $\NC^\dagger(S)$, the following lemma further demonstrates that it is in fact $1$-dense at each $r$-sphere.

\begin{lemma}\label{lem_one_dense}
Let $S$ be a surface as given above. For any $\alpha \in \mathcal{S}_o(r)$, there exists $\alpha' \in \So(r)$ such that $d^\dagger(\alpha,\alpha') = 1$, i.e., $\So(r)$ is 1-dense in $\mathcal{S}_o(r)$.
\end{lemma}
\begin{proof}
Let $o, v_1, v_2, \dots, v_{r-1}, \alpha$ be a geodesic path in $\NC^\dagger(S)$ connecting $o$ to $\alpha$. There is a finite subset $P \subset S$ and $\alpha_1\in \NCo(S)$ both disjoint from $\alpha$ such that $\alpha_1$ is also disjoint from $v_{r-1}$, isotopic to $\alpha$ on $S-P$ and in minimal position with $o$ on $S - P$. By Lemma \ref{radial_dist_pres} and triangle inequality,
$$
r-1=d_o^{\dagger}(o,\alpha)-d_o^{\dagger}(\alpha,\alpha_1) \leq d^{\dagger}(o, \alpha_1)= d_{o}^{\dagger}(o, \alpha_1) \leq d^{\dagger}_o(o,v_{r-1})+d^{\dagger}_o(v_{r-1},\alpha_1)=r.
$$
If $d_{o}^{\dagger}(o, \alpha_1) = r$, then there is nothing to prove. So we assume that $d_{o}^{\dagger}(o, \alpha_1) = r - 1$. 
   
By Lemma \ref{lem_G3}, we can take a $\beta\in \So(r)$ disjoint from both $v_{r-1}$ and $\alpha_1$. Note that $\alpha\cap o$ is not discrete. Since $\alpha$ and $o$ are proper images, we can take a family $\big(U_i\big)_{i=1,\dots, m}$ of small open subsets on $S$ such that $U_i\cap U_j=\emptyset$ for $i\neq j$, $U\coloneqq \bigcup_{i=1}^m U_i$ does not contain the entire $o$ and $(\alpha\cap o)-U$ is discrete, {\it i.e.,} $\alpha$ and $o$ are transverse outside of $U$. Hence we can similarly isotope $\beta$ as in Lemma \ref{lem_pert} so that $\beta$ is transverse to both $o$ and $\alpha$ outside of $U$. Finally, by isotoping $\beta$ away from $U$ and preserving the intersection pattern, we can obtain a $\beta$ that is transverse to both $\alpha$ and $o$. Note that if $U$ is small enough, then $\beta$ can remain in minimal position with $o$ on $S-P$, guaranteeing that $d^{\dagger}_o(o,\beta)=r$ after a quick application of Lemma \ref{radial_dist_pres} and Proposition \ref{prop_NC_approx}.

If $\beta$ is already disjoint from $\alpha$, then $\beta$ will be the $\alpha'$ we are searching for. Otherwise, there exists at least one bigon bounded by subarcs of $\alpha$ and $\beta$ on $S - P$. By isotope the subarc of $\beta$ along the bigon to the other side of $\alpha$ and then perturb the curve, we can obtain a curve $\beta_1$ with strictly fewer intersection points with $\alpha$, disjoint from $\beta$ and remaining transverse to $o$. Since this bigon does not contain any points from $P$, $\beta_1$ stays in minimal position with $o$. Hence by Lemma \ref{radial_dist_pres} and Proposition \ref{prop_NC_approx} again, $\beta_1 \in \So(r)$. Because $\alpha_1\in [\alpha]_{S-P}$ and $\beta\in [\beta]_{S-P}$ are disjoint, we can inductively construct a path $\beta_1, \beta_2, \cdots, \beta_n \subset \So(r)$ until $\beta_n$ and $\alpha$ are in minimal position (\textit{i.e.} no bigon), which in this case will be disjoint. Thus $\alpha' = \beta_n$ is the curve we are looking for.

Finally, we remark that all the perturbation and isotopy can be done on the subsurface $S-\gamma$ if $\alpha$ is disjoint from $\gamma$. This means that $\alpha'$ can be taken disjoint from $\gamma$.
\end{proof}

Note that the perturbation in the proof of Lemma \ref{lem_one_dense} can also be carried out for separating curves. Using the $d_\mathcal{C}^s(S-P)$-distance estimation for curve in minimal position instead of $d_{\mathcal{NC}(S-P)}$, we can conclude the same result for separating curves:
\begin{lemma}\label{lem_one_dense'}
Let $S$ be a surface as given above. For any $\alpha \in \mathcal{C}^\dagger(S)$ with $d^\dagger(o,\alpha)=r>0$, there exists $\alpha' \in \mathcal{C}^\dagger(S)$ transverse to $o$ such that $d^\dagger(\alpha,\alpha') = 1$ and $d^\dagger(o,\alpha') = r$. 
\end{lemma}

Now we are able to conclude Theorem \ref{thm1}.
\begin{proof}[Proof of Theorem \ref{thm1}]
Let $r>0$ be any positive integer. By Lemma \ref{lem_connect}, the sphere $\So(r)$ in $\NCo(S)$ is connected. As $\So(r)$ is $1$-dense in the sphere $\mathcal{S}_o(r)$ in $\mathcal{NC}^\dagger(S)$ by Lemma \ref{lem_one_dense}, $\mathcal{S}_o(r)$ is also connected.
\end{proof}

As an corollary, we can deduce a connectedness of thickened spheres in the fine curve graph as below:
\begin{corollary}
Let $S$ be as given above. Then for any integer $r>0$ and any $o\in\mathcal{C}^\dagger(S)$, then the thickened sphere
$$\Sigma_r\coloneqq\left\{\gamma\in\mathcal{C}^\dagger(S)\colon d^\dagger(o,\gamma)=r,r+1\right\}$$
is connected.
\end{corollary}
\begin{proof}
Let $r>0$ be any positive integer. Let us consider the subgraph
$$\Co(S)\coloneqq\left\{\alpha\in\mathcal{C}^\dagger(S)\colon \alpha\text{ is transverse to }o\right\}\subset \mathcal{C}^\dagger(S)$$
of curves transverse to $o$. By Remark \ref{rem_radii_non_sep}, the thickened sphere in $\Co(S)$
$$\Sigma^\pitchfork_r\coloneqq\left\{\gamma\in\Co(S)\colon d^\dagger(o,\gamma)=r,r+1\right\}$$
is a subgraph of $\Sigma_r$. By the virtue of Lemma \ref{lem_one_dense'}, it suffices to verify that $\Sigma^\pitchfork_r$ is connected.

Indeed, let $\alpha,\beta\in \Co(S)$ be any disjoint pair with $d^\dagger(o,\alpha)=d^\dagger(o,\beta)=r$. Similar to the previous cases, we can find a non-separating curve $\alpha'\in \NCo(S)$ disjoint from $\alpha$ such that $d_Y\big(\pi_Y(o),\pi_Y(\alpha')\big)>M$, where $M$ is the constant from Theorem \ref{thm_surviving_bounded_geodesic_image} and $Y=S-\alpha$. This will make $d^\dagger(o,\alpha')=r+1$, or equivalently $\alpha'\in \So(r+1)$. Similarly, we can also find a $\beta'\in \So(r+1)$ disjoint from $\beta$. Appealing to Lemma \ref{lem_connect}, we can find a path on $\So(r+1)$ connecting $\alpha'$ to $\beta'$. Now the concatenating path
$$\alpha \text{ --- } \alpha'\rightsquigarrow \beta'\text{ --- } \beta$$
verifies the condition \ref{G3'} for $\Co(S)$ with $w=1$. The same arguments can also allow us to verify \ref{G2'} for $\Co(S)$ with $w=2$. Hence Lemma \ref{lem_wright} applies to complete the proof.
\end{proof}

%% file: sections/sec-6.tex
\section{Group action on the boundary}
In this section, let us assume that the curves are only $\mathcal{C}^0$-embedding of $S^1$ into $S$, {\it i.e.} we will be considering the topological version of the fine curve graphs. As per Remark \ref{rem_topological_version}, there is no difference in terms of their boundary at infinity.

Let $S$ be an orientable compact surface as above and let $\homeo(S)$ be the homeomorphism group of $S$. In this section, we will make several observations on the action of the topological group $\homeo(S)$ on the topological space $\partial \mathcal{C}^\dagger(S)$.

We remark that $\homeo(S)$ acts by isometries on the fine curve graph $\mathcal{C}^\dagger(S)$ and thus on the Gromov bordification of $\mathcal{C}^\dagger(S)$. But the action of $\homeo(S)$ on $\mathcal{C}^\dagger(S)$ is not continuous: a sequence of homeomophisms $g_n\to \Id$ in compact-open topology (and even isotopic to $\Id$) can send a simple closed curve $\gamma\subset S$ to $g_n \gamma$ such that $g_n \gamma$ is disjoint from $\gamma$ for all $n$, {\it i.e.} $d^\dagger(g_n \gamma,\gamma)=1$ does not converge to $0$ as $n\to \infty$. But since we are only interested in its geometry at large scale, this does not cause any problem:

\begin{proposition}
The action of $\homeo(S)$ on $\partial \mathcal{C}^\dagger(S)$ is (uniformly) continuous.
\end{proposition}
\begin{proof}
It suffices to show that whenever given a sequence $g_n\to \Id$ as $n\to \infty$, for any $M>0$ and $\xi\in \partial \mathcal{C}^\dagger(S)$, there exists $M'>0$ and $N>0$ such that if $\langle \xi,\eta\rangle_o>M'$, then $\langle \xi,g_n\eta\rangle_o>M$ for any $n>N$. 

Indeed, since the open sets on the boundary is also generated by shadows defined in (\ref{eq_shadow}), we can find an $\alpha\in\mathcal{C}^\dagger(S)$ such that whenever $\eta\in\partial \mathcal{C}^\dagger(S)$ satisfies $\langle \alpha,\eta\rangle_o>R$, we will have $\langle \eta,\xi\rangle_o>M$. Moreover, we can assume that $\langle \alpha,\xi\rangle_o$ is much larger than $R$. 

Since the topology on $\homeo(S)$ is the compact-open topology, there exists $N>0$ such that $g_n \alpha$ is contained in a uniformly small neighbourhood of $\alpha$ for every $n>N$. This implies that for any curve $\gamma\subset S$ that is disjoint from this small neighbourhood and {\it a fortiori} from $\alpha$, $\gamma$ is also disjoint from $g_n \alpha$ for all $n>N$. Hence $d^\dagger(\alpha,g_n\alpha)\leq 2$ for all $n>N$. The same arguments hold for $o$ so that $d^\dagger(o,g_no)\leq 2$ for $n> N$. Let $\eta\in\partial \mathcal{C}^\dagger(S)$ such that $\langle \eta,\xi\rangle_o>M'$. We compute by (\ref{eq2.2}) and (\ref{product_est})
\begin{align*}
    \langle \alpha,g_n\eta\rangle_o&=\langle g_n^{-1}\alpha,\eta\rangle_{g_n^{-1}o}\geq \langle g_n^{-1}\alpha,\eta\rangle_o-2-O(\delta)\\
    &\geq \min \big(\langle\xi,g_n^{-1}\alpha\rangle_o, \langle \xi,\eta\rangle_o\big)-2-O(\delta)\\
    &\geq \min \big(\langle\xi,g_n^{-1}\alpha\rangle_o, M'\big)-2-O(\delta)
\end{align*}
But as $d^\dagger(\alpha,g_n\alpha)\leq 2$, we have $\left|\langle\xi,g_n^{-1}\alpha\rangle_o-\langle\xi,\alpha\rangle_o\right|\leq 2+O(\delta)$, which means that $\langle\xi,g_n^{-1}\alpha\rangle_o$ is also much larger than $R$. Hence by letting $M'$ be large enough, we can assure that $\langle \alpha,g_n\eta\rangle_o>R$, which implies that $\langle \xi,g_n \eta\rangle_o>M$ for all $n>N$.
\end{proof}

Recall that for a $\delta$-hyperbolic space $X$ and a subgroup $G<\Isom(X)$. Let $o\in X$ be a (and thus any) point in $X$. The {\it limit set} of $G$, denoted $\Lambda(G)$, is given by
$$\Lambda(G)\coloneqq \left\{\xi \in \partial X \colon g_n o\to \xi \ \mathrm{for\ some\ }(g_n)\in G^\mathbb{N} \right\}.$$
Now let us examine the limit set of $\homeo(S)$ on the Gromov boundary of $\mathcal{C}^\dagger(S)$.

\begin{lemma}\label{lemma_lim_set}
Let $S$ be an orientable surface as above. The limit set of $\homeo(S)$ on $\partial \mathcal{C}^\dagger(S)$ is $\partial \mathcal{C}^\dagger(S)$.
\end{lemma}
\begin{proof}
Let us consider $\NC^\dagger(S)$, which is quasi-isometric to $\mathcal{C}^\dagger(S)$. We remark that the homeomorphism $\partial \NC^\dagger(S)\to \partial \mathcal{C}^\dagger(S)$ is $\homeo(S)$-equivariant. So it suffices to show that the limit set of $\homeo(S)$ on $\partial \NC^\dagger(S)$ is $\partial \NC^\dagger(S)$ itself. Indeed, by the theorem of classification of finite-type surfaces, for any two non-separating simple closed curves $\alpha,\beta\subset S$, the bordered surfaces $S-\alpha$ and $S-\beta$ are homeomorphic, so there exists $\phi\in \homeo(S)$ such that $\phi(\alpha)=\beta$, {\it i.e.} the action of $\homeo(S)$ on $\NC^\dagger(S)$ is transitive. Let $o\in \NC^\dagger(S)$ be any base point and $\xi\in \partial \NC^\dagger(S)$. If $\alpha_n$ is a Cauchy-Gromov sequence converging to $\xi$, then by transitivity of the action, for each $n>0$, there exists $\phi_n\in \homeo(S)$ such that $\phi_n(o)=\alpha_n$. But this means that $\phi_n(o)\to \xi$. Hence we can conclude that $\Lambda\big(\homeo(S)\big)$ is the entire boundary.
\end{proof}

Recall that an isometric group action on a $\delta$-hyperbolic space $X$ is of {\it general type} if there exists two independent hyperbolic isometries, or equivalently if the action admits no fixed point on $\partial X$. For isometric actions of general type on $\delta$-hyperbolic spaces, there is an interesting dynamical property:
\begin{proposition}[Corollary 7.4.3 (ii), \cite{das2017geometry}]\label{prop_limit_set}
Let $X$ be a $\delta$-hyperbolic space and let $G$ act on $X$ by isometries. Suppose in addition that the action is of general type. Then the limit set $\Lambda(G)$ is the smallest non-empty closed $G$-invariant subset of $\partial X$.
\end{proposition}

Recall that a topological space $X$ on which $G$ acts is said {\it $G$-minimal} if for every point $x\in X$, the closure of the orbit $\overline{Gx}$ is $X$. Such an action will also be called a {\it minimal action}.

\begin{proposition}
Let $S$ be an orientable surface as above. Then the action of $\homeo(S)$ on $\partial \mathcal{C}^\dagger(S)$ is minimal.
\end{proposition}
\begin{proof}
In view of Proposition \ref{prop_limit_set} and Lemma \ref{lemma_lim_set}, it suffices to show that the action of $\homeo(S)$ on $\mathcal{C}^\dagger(S)$ is of general type. To prove so, it suffices to take two independent pseudo-Anosov homeomorphisms on $S$, which will yield two independent hyperbolic isometries on $\mathcal{C}^\dagger(S)$ as their action of a subgraph of $\mathcal{C}^\dagger(S)$ isometric to $\mathcal{C}(S)$ is of general type (see \cite{masur1999geometry}).
\end{proof}

\begin{remark}
The action of $\mathrm{Homeo}(S)$ on $\partial \mathcal{C}^\dagger(S)$ is not transitive since the action cannot send the points on $\partial \mathcal{C}^\dagger(S)$ that are represented by ending laminations to those who are not.
\end{remark}

To conclude this section, we wish to point out in the interests of curiosity that $\homeo(S)$ is a Polish group, {\it i.e.} a completely metrizable and separable topological group, and its continuous minimal actions on compacta are of special interests. For example, \cite{gutman2021universal} shows that its {\it universal minimal flow}, a compact space with a minimal $\homeo(S)$-action that admits any compact minimal space for $\homeo(S)$-actions as a factor, is not metrizable.